\documentclass[a4paper,reqno]{amsart}

\usepackage{amsmath}
\usepackage{paralist}
\usepackage{graphics}
\usepackage{epsfig}
\usepackage{amsfonts}
\usepackage{amssymb}
\usepackage{hyperref}
\usepackage{epstopdf}
\usepackage{color}

\newtheorem{theorem}{Theorem}[section]
\newtheorem{lemma}[theorem]{Lemma}

\def\ifl{\iffalse }

\def\bc{\begin{center}}       \def\ec{\end{center}}
\def\ba{\begin{array}}        \def\ea{\end{array}}
\def\be{\begin{equation}}     \def\ee{\end{equation}}
\def\bea{\begin{eqnarray}}    \def\eea{\end{eqnarray}}
\def\beaa{\begin{eqnarray*}}  \def\eeaa{\end{eqnarray*}}

\numberwithin{equation}{section}

\newtheorem{remark}[theorem]{Remark}
\numberwithin{equation}{section}

\begin{document}

\title[Fully parabolic chemotaxis system,   boundedness, large time behavior]
{On boundedness,  blow-up and convergence in a  two-species and two-stimuli  chemotaxis system with/without loop}

\author{Ke Lin}
\address{School of Economics and Mathematics, Southwestern University of Economics and Finance,Chengdu, 610074 SICHU China}
\email{linke@swufe.edu.cn}

\author{Tian Xiang$^*$}
\address{Institute for Mathematical Sciences, Renmin University of China, Bejing, 100872, China}
\email{txiang@ruc.edu.cn}
\thanks{$^*$ Corresponding author.}




\begin{abstract}
In this work, we  study  dynamic properties of classical solutions to a  homogenous  Neumann initial-boundary value problem (IBVP) for a two-species and two-stimuli  chemotaxis model with/without chemical signalling loop in a 2D  bounded and smooth domain.  We detect the product of two species masses  as a feature to determine boundedness, gradient estimate, blow-up and exponential convergence of classical solutions for the corresponding IBVP. More specifically, we first show generally a   smallness on  the  product of both species masses, thus allowing one species mass to be  suitably large, is sufficient to guarantee  global boundedness, higher order gradient estimates and $W^{j,\infty}(j\geq 1)$-exponential convergence  with rates of convergence to constant equilibria; and then, in  a special case, we detect a straight  line of masses on which blow-up occurs for  large product of masses. Our findings provide  new understandings about the underlying model, and thus,  improve and extend greatly the existing knowledge relevant to this model.
\end{abstract}

\maketitle

\section{Introduction and statement of main results}

In this work, we  further study  dynamic properties of classical solutions to the Neumann initial-boundary value problem for the  following two-species and two-stimuli  chemotaxis model with/without chemical signalling loop:
\be\label{Loop equations}\begin{cases}
u_t =  \nabla\cdot\left(\nabla u-\chi_1u\nabla v\right) &\text{in } \Omega\times(0,\infty), \\[0.2cm]
\tau_1 v_t = \Delta v-v+w  &\text{in } \Omega\times(0,\infty), \\[0.2cm]
w_t =  \nabla\cdot\left(\nabla w-\chi_2w\nabla z-\chi_3w\nabla v\right) &\text{in } \Omega\times(0,\infty), \\[0.2cm]
\tau_2 z_t = \Delta z-z+u&\text{in } \Omega\times(0,\infty), \\[0.2cm]
\frac{\partial u}{\partial \nu}=\frac{\partial v}{\partial \nu}=\frac{\partial w}{\partial \nu}=\frac{\partial z}{\partial \nu}=0&\text{on } \partial\Omega\times(0,\infty), \\[0.2cm]
\left(u, \  \tau_1v,\  w, \   \tau_2z\right)=\left(u_0,\    \tau_1v_0,\    w_0, \  \tau_2z_0\right)&\text{in } \Omega\times\{0\}.
\end{cases}
\ee
 Here, $\Omega\subset \mathbb{R}^2$ is  a bounded and smooth domain  and $\frac{\partial}{\partial\nu}$ denotes the outer normal derivative on the boundary  $\partial \Omega$,  $u=u(x,t)$ and $w=w(x,t)$ respectively denote the unknown density of  macrophages and  tumor cells,  while $v=v(x,t)$ and $z=z(x,t)$ represent the concentration of chemical signals secreted by $w$ and $u$, respectively. The modelling parameters $\chi_i>0,\tau_i\geq 0\ (i=1,2)$ and $\chi_3\in \mathbb{R}$ are given constants.

Model \eqref{Loop equations} involves four unknown variables $u,v,w,z$ and describes a two-species and two-stimuli  chemotaxis model with/without chemical signalling loop, depending on $\chi_3=0$ or not:  macrophages $u$ secrete a chemical signal $z$, called gradient epidermal growth factor, which has an attractive impact on tumor cells $w$ and further stimulates them to secrete the other chemical signal $v$, called the colony stimulating factor 1, which attracts  macrophages $u$ to aggregate and binds to receptors of the macrophages $u$,
continuing the activation of them in return \cite{A1}. This model contains two widely-studied sub-models: upon setting $u=z\equiv0$ or $\chi_1=\chi_2=0$, the well-known one-species and one-stimuli minimal Keller-Segel model follows:
\be \label{KS-1}\begin{cases}
\tau_1 v_t = \Delta v-v+w  &\text{in } \Omega\times(0,\infty), \\[0.2cm]
w_t =  \nabla\cdot\left(\nabla w-\chi_3w\nabla v\right) &\text{in } \Omega\times(0,\infty).
\end{cases}
\ee
 This minimal KS model is well-known to exhibit  critical mass blow-up striking future in 2D (small mass $\chi_3\|w_0\|_{L^1(\Omega)}<\pi^*$ , defined by Lemma \ref{TM-ineq} below, yields boundedness \cite{GZ98, A12}, otherwise, blow-up may occur \cite{HV97, HW01, JL92, A11}) and generic blow-up occurs in $\geq 3$D, see the review articles  \cite{A16, Horst-03, Win10-JDE, Win2013} for more information. The second important sub-model is the following two-species and two stimuli chemotaxis model  obtained by setting $\chi_3=0$:
 \be\label{KS-2}\begin{cases}
u_t =  \nabla\cdot\left(\nabla u-\chi_1u\nabla v\right) &\text{in } \Omega\times(0,\infty), \\[0.2cm]
\tau_1 v_t = \Delta v-v+w  &\text{in } \Omega\times(0,\infty), \\[0.2cm]
w_t =  \nabla\cdot\left(\nabla w-\chi_2w\nabla z\right) &\text{in } \Omega\times(0,\infty), \\[0.2cm]
\tau_2 z_t = \Delta z-z+u&\text{in } \Omega\times(0,\infty).
\end{cases}
\ee
 \textbf{When $\tau_1=\tau_2=0$}, Tao and Winkler \cite{A4}  systematically studied the boundedness vs blow-up, wherein $\chi_1$ and $\chi_2$ are allowed to be real: for either $\chi_1<0$ or $\chi_2<0$, boundedness for large initial data is guaranteed in $\leq 3$D; in the challenging while more interesting case when  both $\chi_1>0$  and $\chi_2>0$,  boundedness vs blow-up
is characterized by  the total mass of both  species: writing
\be\label{mass-def}
m_1=\int_{\Omega}u_0, \quad \quad m_2=\int_{\Omega}w_0,
\ee
then boundedness is ensured for $\max\{m_1, \   m_2\}<C_0$ with some $C_0>0$,  whereas,  for $\chi_1=\chi_2=1$, finite time blow-up in $2$D may occur for $\min\{m_1, \   m_2\}>4\pi$. These results were  improved by Yu et. al. in \cite{A3} by showing that $C_0=4\pi$ and a blow-up criterion that
\be\label{blowup-yu}
\frac{1 }{m_2\chi_1}+\frac{1 }{m_1\chi_2}<\frac{1}{2\pi}.
 \ee
 Very recently, we observed in \cite{A2}  that the chemotactic signaling loop between two cell types  bridges  certain relationship between $u$ and $w$, and therefore,  the dynamics of one species shall be essentially determined by the other. To verify that, we considered  a 2D much simplified version of \eqref{Loop equations} in the unit ball $\Omega =B_1(0)\subset\mathbb{R}^2$ with the second and fourth equation respectively replaced by
\be\label{S-KS}
0=\Delta v-\bar{w}_0+w, \quad \quad 0=\Delta z-\bar{u}_0+u,  \quad \quad \bar{u}_0=\frac{1}{|\Omega|}\int_\Omega u_0.
\ee
In this setup,  the problem essentially becomes two  1D scalar parabolic equations, which renders parabolic comparison principles applicable. Then substantial progresses  on the simultaneous  boundedness and finite-time blow-up are provided and, in particular, the previous boundedness for  both small masses was improved to be $\min\left\{m_1\chi_2, \   m_2\chi_1\right\}<4\pi$, requiring only one mass be small. While, those arguments, especially \cite[Lemma 3.1]{A2},  seem to be hardly adapted to \eqref{KS-2} even in radial settings. Even through suitable largeness of both masses are known to produce blow-ups \cite{A2, A4, A3} (c.f. also \eqref{blowup-yu}), however, in non-radial settings, as a starting motivation of this project, we are wondering
\begin{itemize}
 \item[(Q1)] whether suitable largeness of one mass is still able to ensure boundedness and further convergence?
\end{itemize}

 \textbf{When $\tau_1>0, \tau_2>0$}, in this fully parabolic case, much less seems to be known except that Li and Wang \cite{A5} provided boundedness for \eqref{KS-2} under an implicit smallness condition on both $m_1$ and $m_2$. On the other hand,   to our best knowledge,  so far,  no blow-up has been detected yet and there seems even no available result on large time behavior of bounded solutions to either \eqref{KS-2} or \eqref{Loop equations}, except adding certain damping sources of Logistic type \cite{TB17, QG19, TW12-non, TMZK18, ZLY17, Zhang19, ZC17} for similar systems. The knowledge  is  far from ideal compared to the case of $\tau_1=\tau_2=0$.  Our second and primary motivation is thus to explore   (Q1) for  both \eqref{KS-2} and \eqref{Loop equations} without any damping sources, and, moreover, we are wondering how much the blow-up criterion  \eqref{blowup-yu} in the elliptic case can be carried over to the fully parabolic case by asking
 \begin{itemize}
   \item[(Q2)]  whether suitable largeness of both masses induces blow-up?
\end{itemize}

  As  a continuation of mainly  works \cite{A5, A2,A4, A3}, our  purpose is to provide further understandings about global dynamics of the two-species and two-stimuli chemotaxis model \eqref{Loop equations} with/without signal loop motivated by the non-obvious questions (Q1) and (Q2) for the cases of $\tau_1=\tau_2=0$ and $\tau_1, \tau_2>0$.  Roughly, going far beyond (Q1) and (Q2), our findings first show that  only a   smallness of product $m_1m_2\chi_1\chi_2$ is needed to ensure global boundedness, higher order gradient estimates and $W^{j,\infty}(j\geq 1)$-convergence with  rates of convergence;  and then, in  a special case, we  detect a line of $m_1$ and $m_2$ on which  blow-up occurs  for  large product of masses. To state  our main results precisely,  we first note from the 2D Gagilardo-Nirenberg interpolation inequality, cf. \eqref{GN-inequ},  there exists $C_{GN}=C_{GN}(\Omega)>0$ such that
  \be\label{GN-inequ0}
\|\psi\|_{L^4(\Omega)}^4\leq 8C_{GN}^4 \|\psi\|_{L^2(\Omega)}^2\|\nabla \psi\|_{L^2(\Omega)}^2+  8C_{GN}^4\|\psi\|_{L^2(\Omega)}^4, \ \ \  \forall  \psi\in W^{1,2}(\Omega).
\ee
Next, thanks to the 2D Sobolev embedding  $W^{1,1}(\Omega)\hookrightarrow L^2(\Omega)$, we define
\be\label{k-def}
k=4|\Omega|\left(\inf\left\{\frac{\|\nabla \varphi\|_{L^1(\Omega)}^2}{\|\varphi-1\|_{L^2(\Omega)}^2}: \ \ \varphi\in W^{1,1}(\Omega),\ \varphi\not\equiv 1, \  \bar\varphi=1\right\}\right)^{-1},
\ee
which is well-defined and is a positive and finite number.  Finally, for the initial data,  we assume, throughout this paper, that
\be\label{initi data reg}
\begin{cases} \left(u_0,\tau_1 v_0, w_0, \tau_2 z_0\right)\in C^0(\bar{\Omega})\times W^{1,\infty}(\Omega)\times C^0(\bar{\Omega})\times W^{1,\infty}(\Omega),\\[0.2cm]
 u_0\geq,\not\equiv0, \  \  \tau_1 v_0\geq0, \  \  w_0\geq,\not\equiv0, \  \  \tau_2 z_0\geq 0. \end{cases}
\ee
With these preparations,  our achievements can be stated precisely as follows.
\begin{theorem}\label{main theorem}
Let $\chi_i>0,\tau_i\geq 0\ (i=1,2), \chi_3\in\mathbb{R}$,  $\Omega\subset \mathbb{R}^2$ be a bounded and smooth domain, and let the initial data    $(u_0,\tau_1 v_0, w_0, \tau_2 z_0)$ fulfill \eqref{initi data reg}, and finally, let $m_1$ and $m_2$ be the respective mass of $u_0$ and $w_0$ as defined by \eqref{mass-def}.
 \begin{enumerate}
 \item[(B1)] [\textbf{Uniform Boundedness}] Assume that
\be\label{small mass bdd}
m_1m_2\chi_1\chi_2<\begin{cases} \left(\pi^*-m_2\chi_3\right)\pi^*, &\text{ if } \tau_1=\tau_2=0,\\[0.2cm]
 \frac{\sqrt{1-4m_2\chi_3C_{GN}^4}}{4C_{GN}^8}, &\text{ if } \tau_1, \tau_2>0. \end{cases}
\ee
 Then the IBVP  \eqref{Loop equations} admits a unique global-in-time  classical solution $(u,v,w,z)$ which is positive and is uniformly bounded in time according to
\be\label{bdd-thm-fin}
\quad \quad  \|u(t)\|_{L^{\infty}}+\|v(t)\|_{W^{1,\infty}}
+\|w(t)\|_{L^{\infty}}+\|z(t)\|_{W^{1,\infty}}\leq C_1,\  \ t\geq 0.
\ee
\item[(B2)][\textbf{Gradient Estimates}]  If either \eqref{small mass bdd} holds or $\|(u\ln u)(t)\|_{L^1}+\|(w\ln w)(t)\|_{L^1}$ is uniformly bounded on the maximal existence time interval  $(0, T_m)$, then $T_m=\infty$ and, besides \eqref{bdd-thm-fin},  the following higher order gradient estimate away from $t=0$, say $t\geq1$,  holds:
\be\label{gradUW-improve-thm}
\begin{split}
\quad\quad \quad &\| \left(u(t), \  w(t)\right)\|_{W^{2,4}}+\|\left(u(t), \ w(t)\right)\|_{W^{1,\infty}}+\|\left(v(t),
\ z(t)\right)\|_{W^{3,\infty}}\\[0.2cm]
 &+\| \left(I_{\{\tau_1=0\}}v(t), \ I_{\{\tau_2=0\}}z(t)\right)\|_{W^{4,4}}\leq C_2, \quad  \forall t\geq 1.
\end{split}
\ee
\item[(B3)][\textbf{$W^{j,\infty}$-Exponential Convergence}] When $\tau_1=\tau_2=0$, assume
\be\label{Exp-con-pe}
 k^2m_1m_2\chi_1\chi_2+km_2|\Omega|\chi_3^+< 4|\Omega|^2, \quad  \chi_3^+=\max\{\chi_3, 0\};
\ee
when $\tau_1>0$ and $\tau_2>0$, assume
\be\label{Exp-con-pp}
\begin{cases}
\frac{2-\sqrt{22}}{3}<\frac{km_2\chi_3}{|\Omega|}<\sqrt{2}, \\[0.2cm] k^2m_1m_2\chi_1\chi_2<\frac{2\sqrt{2}}{3}|\Omega|^2\min\left\{1, \ \ \frac{3}{2}+\frac{km_2\chi_3}{|\Omega|}\right\}.
\end{cases}
\ee
Then the unique global solution of \eqref{Loop equations} decays exponentially according to:
\be\label{lt-thm}
 \ \ \begin{cases}
 \|\left(u(t)-\bar{u}_0, \   w(t)-\bar{w}_0\right)\|_{W^{1, \infty}}\leq C_3 e^{-\frac{\sigma(k)}{14}t}, \   \forall t\geq 1, \\[0.2cm]
 \|\left(v(t)-\bar{w}_0, \  z(t)-\bar{u}_0\right)\|_{W^{3,\infty}}\leq C_4 e^{-\frac{\mu(k)}{44}t}, \  \forall t\geq 1,  &\text{if } \tau_1=\tau_2=0, \\[0.2cm]
 \|\left(v(t)-\bar{w}_0, \  z(t)-\bar{u}_0\right)\|_{W^{2,\infty}}\leq C_5 e^{-\frac{\zeta(k)}{15}t}, \  \forall t\geq 1,  &\text{if } \tau_1, \tau_2>0.
\end{cases}
\ee
\item[(B4)][\textbf{Finite time Blow-up}] Let $\tau_1= \tau_2$ and $\chi_3=0$. Then on the straight  line $m_1\chi_2=m_2\chi_1$,  there exists a family of initial data $(u_0,\tau_1v_0, w_0, \tau_2z_0)$ with
\be\label{large mass blow-up}
m_1m_2\chi_1\chi_2>(\pi^*)^2,
\ee
such that for some finite  $T>0$ the  unique solution of the  IBVP  \eqref{Loop equations} exists  classically  on $\Omega\times (0, T)$ but blows up at $t=T$ in the sense that
\be\label{blowup-new}
\limsup_{t\nearrow T}\left(\|\left(u\ln u\right)(t)\|_{L^1}+\|\left(w\ln w\right)(t)\|_{L^1}\right)=\infty.
\ee
\end{enumerate}
Here and below,   $\pi^*$ is an explicit positive number defined in Lemma  \ref{TM-ineq}, $C_{GN}$  and  $k$ are defined in \eqref{GN-inequ0} and \eqref{k-def}, respectively,  $I_{\{\tau_1=0\}}=1$ if $\tau_1=0$, otherwise, it is zero; $\zeta(k)=\min\{\frac{1}{\tau_1}, \  \frac{1}{\tau_2},\    \frac{\sigma(k)}{2}\}$ and $\sigma(k)=\mu(k)$ if $\tau_1=\tau_2=0$ and $\sigma(k)=\delta(k)$ if $\tau_1,  \tau_2>0$ with $\mu(k)$ defined by \eqref{decay-rate1} and $\delta(k)$ by \eqref{decay-rate}, both of them are functions of $m_i,\chi_i, \tau_i$ and $k$, $C_i=C_i(u_0, \tau_1 v_0, w_0, \tau_2z_0, \chi_i, |\Omega|)$ are positive constants, $\bar{u}_0$, the average of $u_0$,  is defined in \eqref{S-KS}, similarly for $\bar{w}_0$, vector notation is understood as component-wise like $\|(u,v)\|_{L^1}=\|u\|_{L^1}+\|w\|_{L^1}$, and, finally, the commonly abbreviated notations are used: for instance, for a generic function $f$,
$$
\|f(t)\|_{L^p}=\|f(\cdot, t)\|_{L^p}=\|f(\cdot, t)\|_{L^p(\Omega)}=\left(\int_\Omega |f(x,t)|^pdx\right)^\frac{1}{p}.
$$
\end{theorem}
\begin{remark}[Product of masses on boundedness, blow-up and convergence]\label{thm-rem}
\
\
\begin{itemize}
\item[(P1)] Our higher order gradient estimate (B2) and  the $W^{j,\infty}(j\geq 1)$-convergence in (B3) seem to appear the first time in chemotaxis-related systems.
\item[(P2)] The feature of our main results is that we detect a smallness on the product $m_1m_2\chi_1\chi_2$ to ensure boundedness, higher order gradient estimates  and exponential convergence with  rates of  convergence  thanks to (B2) and (B3). These together with Remark \ref{4-2-lt} below show that   our results also extend and improve existing boundedness and convergence of solutions to the one-species and one stimuli Keller-Segel  model \eqref{KS-1}, c.f. \cite{GZ98, Win10-JDE}.
\item[(P3)] When $\chi_3=0$, our boundedness  (B1)  under the (explicit) smallness \eqref{small mass bdd} improves greatly  the existing boundedness under smallness of both $m_1$ and $m_2$ \cite{A5, A4, A3}. When $\chi_3\neq 0$, boundedness  results (B1)-(B3)  extend our previous one  under  radial and elliptic simplification (cf. \eqref{S-KS}) in \cite{A2}. Moreover, the sign effect of $\chi_3$ is exhibited, showing damping effect of repulsion on  boundedness and convergence, especially, when $\tau_1=\tau_2=0$.
\item[(P4)] No matter $\tau_1=\tau_2=0$ or $\tau_1, \tau_2>0$,   (B4)  detects a blow-up line on $m_2\chi_1=m_1\chi_2$. We should point out that it  (along with (B1)) is not a blow-up criterion, while, it gives a rough lower bound for $C_{GN}$, i.,e, $C_{GN}^4\geq \frac{1}{2\pi^*}$. Also, in convex domains,  $k$ has a lower bound: $k\geq \frac{4d^2}{\pi^2}$, cf. Lemma \ref{k-const}. In non-symmetric domains, i.e., $\pi^*=4\pi$,  it is easy to see that this blow-up line is inside  the range enclosed by  the blow-up criterion  \eqref{blowup-yu},  which  is inside  the range enclosed by \eqref{large mass blow-up}. Hence, (B1) is not optimal even in the case of $\tau_1=\tau_2=0$. Together with (B1), we see that the critical curve that distinguishes boundedness and blow-up for \eqref{KS-2} must contain $(m_2\chi_1,\  m_1\chi_2)=(\pi^*, \ \pi^*)$ as a boundary point. We conjecture that a general version of  the blow-up criterion \eqref{blowup-yu} for \eqref{KS-2} continues to hold  in the case of $\tau_1, \tau_2>0$, namely,
     $$
    \frac{1 }{m_2\chi_1}+\frac{1 }{m_1\chi_2}<\frac{2}{\pi^*}.
    $$
    We leave this open problem as a future investigation.
\end{itemize}
\end{remark}

The point of our project is that we detect the product $m_1m_2\chi_1\chi_2$   as a feature to determine boundedness, gradient estimate, blow-up and $W^{j,\infty}(j\geq 1)$-exponential convergence for \eqref{Loop equations}. First,  the smallness of  $m_1m_2\chi_1\chi_2$ in \eqref{small mass bdd} or  \eqref{Exp-con-pe} or \eqref{Exp-con-pp}  allows us to choose suitably large  mass of one species  and small for the  other to ensure   global boundedness, higher order gradient estimates or  exponential convergence.   This  is in sharp contrast to  those of \cite{A5, A4, A3}, wherein smallness of both masses are needed to have boundedness in the sense of \eqref{bdd-thm-fin}. Second,  we  find a line of masses on which  blow-up  occurs for large product of masses. While, we have to mention that, even in the case $\tau_1=\tau_2=0$, critical mass phenomenon for \eqref{KS-2} has not been detected yet.  Critical mass phenomenon does  exist in  the minimal classical KS model with one-species and one-stimuli \cite{JL92, A11,A12,A13}, while, for more complex or multi-species chemotaxis systems, boundedness and blow-up exist \cite{ESV09, A6, A5, A2, A14,  A4, A3},  but critical mass blow-up occurs rarely \cite{A7, JW16-JDE}. In a future exploration, we shall aim to  determine a critical curve which separates  bounedeness and blow-up  for \eqref{Loop equations} or its simplified version like \eqref{KS-2}.

 In the remaining of this section,  we outline the structure of this paper, which contains four main sections.

In the present section, we provide an introduction to our two-species and two-stimuli  chemotaxis model with/without chemical signalling loop  that encompasses two important widely-studied sub-models, and then we formulate our main motivations and state our main findings in  Theorem \ref{main theorem} on product of masses on boundedness, higher order gradient estimates, blow-up  and exponential convergence.

 In Section 2, we  first state the local existence and extensibility of smooth solutions to the IBVP \eqref{Loop equations}, and then,  we obtain a standard $W^{1,q}$-estimate for an inhomogeneous  heat/elliptic equation, cf. Lemma \ref{reciprocal-lem}, and then, upon an observation of best constant for the Poincar\'{e} inequality  \cite{PW60}, we find  an explicit lower bound for $k$ defined in \eqref{k-def} in convex domains, cf. Lemma \ref{k-const}, and, finally, for convenience,  we state a version of Trudinger-Moser inequality \cite{ A12} and the widely-used 2D Gagliardo-Nirenberg  inequality \cite{Fri-bk}, which will be used later on.

 To make the flow of our ideas more clear, we divide Section 3 into 4 subsections to  prove our stated  boundedness,  gradient estimates  and finite time blow-up  in (B1),  (B2) and (B4), respectively. Our analysis begins with a general identity associated with \eqref{Loop equations} which becomes a conditional Lyapunov functional in the case of $\tau_1=\tau_2=0$ and small product of masses, cf. Lemmas \ref{Lyapunov-f-lemma} and \ref{u+w-log-bdd-lem}, and thus yields the key starting uniform $L^1$- boundedness of $(u\ln u, w\ln w)$. In the fully parabolic case, the same boundedness is derived based on estimating the differential of a well-selected combined energy together with subtle analysis, cf. Lemma \ref{fully PP}. Then, using quite known testing procedure, we raise the obtained $L^1$-boundedness  of $(u\ln u, w\ln w)$ eventually to the one stated in \eqref{bdd-thm-fin}, cf. Lemma \ref{linfty-bdd-lem}. Right after that, we are devoted to showing the uniform  $L^1$-boundedness  of $(u\ln u, w\ln w)$  indeed implies higher order gradient estimates as in (B2). To achieve that goal, we progressively use energy method together with the 2D G-N interpolation inequality, $W^{2,p}$-elliptic and $W^{1,q}$-parabolic estimates to derive  uniform estimates  for the following route map of mainly $u$ (similar for  $w$):
 $$
 \|\nabla u\|_{L^2}\rightarrow  \|\nabla u\|_{L^4}\rightarrow  \|\Delta u\|_{L^2}+\|(v,z)\|_{W^{2,\infty}}\rightarrow  \|\Delta u\|_{L^4}+\|(v,z)\|_{W^{3,\infty}}.
 $$
  Finally, in Subsection 3.4, on the straight line $m_2\chi_1=m_1\chi_2 $, we construct initial data satisfying \eqref{large mass blow-up} so that  the corresponding solution  of \eqref{Loop equations} blows up in finite time according to \eqref{blowup-new} due to (B2). In this case, upon an observation that  our two-species and two-stimuli model \eqref{Loop equations} is a two-copy of the one-species and one-stimuli minimal KS model (cf. Lemma \ref{link}), we make use of  the well-known blowup knowledge about the minimal model (\cite{HV97, HW01, Horst-03, JL92, A11,A12,A13}) to construct the existence of finite time blow-up for \eqref{Loop equations} as in (B4), cf. Lemma \ref{blowup-lem}.

  In Section 4, inspired from  \cite{GZ98}, we first transform our model \eqref{Loop equations} conveniently into an equivalent one in \eqref{equal system}, and then we construct two well-chosen testing  functionals involving $(U\ln U, W\ln W)$  in \eqref{F-def-con-pe} and \eqref{energy function},  which become genuine Lyapunov functionals and decay exponentially with precise rates under \eqref{Exp-con-pe} or \eqref{Exp-con-pp}, cf. Lemmas \ref{con-lem-pe} and \ref{con-lem-pp}. Here,  $(U, \ W)=(\frac{u}{\bar{u}_0},\ \frac{w}{\bar{w}_0})$, cf. \eqref{trans-id}.  As consequences, we obtain the crucial starting  $L^1$-exponential convergence of $(U\ln U, W\ln W)$ with precise convergence rates. Then with the aid of  the Csisz\'{a}r-Kullback-Pinsker inequality (cf. \cite{CJMTU01}), we indeed obtain $L^p \ (p\geq 1)$-exponential convergence of $(U-1, W-1)$, cf. Lemma \ref{con-improve-UWLp}. With these information at hand, one can use (commonly used,  cf. \cite{A19, A20, Win10-JDE}) the standard $W^{2,p}$-estimate in the case of $\tau_1=\tau_2=0$ or the $L^p$-$L^q$-smoothing estimate for the Neumann heat semigroup $\{e^{t\Delta}\}_{t\geq0}$ in the case of $\tau_1, \tau_2>0$ to derive the exponential decay of bounded solutions in up to $L^\infty$-norm. Here, thanks to our uniform higher order gradient estimates as in  (B2), instead, we readily utilize  the G-N interpolation inequality to  improve the $L^p$-convergence to $W^{j,\infty} (j\geq 1)$-convergence of $(U, V, W, Z)$ with  rate of convergence. Upon simple translations, we  achieve the $W^{j,\infty} $-exponential convergence for our original  model \eqref{Loop equations} in Lemma \ref{con-fin-lem}, which is indeed more than what has been stated in (B3).

\section{Preliminaries  and basic results on our model}

 We first state the well-established local  well-posedness  and extensibility  of solutions to the IBVP \eqref{Loop equations} and  elementary $L^1$-properties of  local solutions.
\begin{lemma}\label{2.l1}
Let $\chi_i>0,\tau_i\geq 0\ (i=1,2), \chi_3\in\mathbb{R}$,  $\Omega\subset \mathbb{R}^2$ be a bounded and smooth domain, and, let the initial data $(u_0,\tau_1 v_0, w_0, \tau_2 z_0)$ fulfill \eqref{initi data reg}.  Then there exist a maximal existence time $T_m\in (0,\infty]$ and a uniquely determined pair of positive  functions $(u,v,w,z)\in \left(C(\bar{\Omega}\times [0, T_m))\cap C^{2,1}(\bar{\Omega}\times (0, T_m))\right)^4$ that solve  the IBVP \eqref{Loop equations} classically on   $\Omega\times (0, T_m)$ and fulfill the following  $L^1$-properties: for $t\in (0, T_m)$,
\be\label{uvwz-l1}
\begin{cases}
  \|u(t)\|_{L^1}=\|u_0\|_{L^1}, \ \  \|w(t)\|_{L^1}=\|w_0\|_{L^1} ,\\[0.25cm]
 \|v(t)\|_{L^1}=\|w_0\|_{L^1}+\begin{cases}
 0, & \text{ if } \tau_1=0, \\[0.2cm]
 \left(\|v_0\|_{L^1}-\|w_0\|_{L^1}\right)e^{-\frac{t}{\tau_1}}, & \text{ if } \tau_1>0,
 \end{cases} \\[0.2cm]
 \|z(t)\|_{L^1}=\|u_0\|_{L^1}+\begin{cases}
 0, & \text{ if } \tau_2=0, \\[0.2cm]
 \left(\|z_0\|_{L^1}-\|u_0\|_{L^1}\right)e^{-\frac{t}{\tau_2}}, & \text{ if } \tau_2>0.
 \end{cases}
\end{cases}
\ee
 Moreover, the local solution $(u,v,w,z)$ fulfills the following extensibility criterion:
\be\label{u-w blow-up criterion}
T_m<\infty \Rightarrow  \limsup_{t\nearrow T_m}
\left(\|\left(u(t), w(t)\right)\|_{L^{\infty}}+\|\left(\tau_1 v(t), \tau_2 z(t)\right)\|_{W^{1,\infty}}\right)=\infty.
\ee
\end{lemma}
\begin{proof}
  The local  well-posedness  and extensibility  of solutions to the IBVP \eqref{Loop equations} and thus  \eqref{equal system} have been well-established  via Banach contraction principle and parabolic regularity; see e.g.   \cite{A16,A21, A4, STW14,TW11,  ZLY17} for closely-related chemotaxis systems. The conservations of $u$ and $w$ follow upon integration by parts on the first and third equation in \eqref{Loop equations}. By a simple integration of the $v$-equation and using the homogeneous Neumann boundary conditions, one has
  $$
  \tau_1\frac{d}{dt}\int_\Omega v+\int_\Omega v=\int_\Omega w=\int_\Omega w_0,
  $$
  which implies the  identity for $L^1$-norm of $v$ in \eqref{uvwz-l1}. Likewise, the identity for $L^1$-norm of $z$ in \eqref{uvwz-l1} follows.   \end{proof}

\begin{lemma}  \label{reciprocal-lem} Let $\Omega\subset \mathbb{R}^2$ be a bounded and smooth domain and let
\be\label{q-exp}
\begin{cases}
q\in [1, \frac{2p}{2-p}), \quad & \text{if }  1\leq p\leq 2,\\
q\in [1, \infty], \quad & \text{if }  p>2.
\end{cases}
\ee
Then there exist  $C_1=C_1(p,q, \tau_1v_0, \Omega)>0$ and $C_2=C_2(p,q, \tau_2z_0, \Omega)>0$ such that  the unique local-in-time classical solution $(u,v,w,z )$ of \eqref{Loop equations} satisfies
 \be\label{gradvz-bdd-lp}
\begin{cases}
\| v(t)\|_{W^{1,q}} \leq C_1\left(1+\sup_{s\in(0,t)}\|w(s)\|_{L^p}\right), \ \ \ \forall t\in(0, T_m), \\[0.25cm]
\| z(t)\|_{W^{1,q}} \leq C_2\left(1+\sup_{s\in(0,t)}\|u( s)\|_{L^p}\right), \ \ \ \forall t\in(0, T_m).
\end{cases}
\ee
In particular, for any $q\in[1,2)$, there exists $C_3=C_3(q, \tau_1v_0, \tau_2z_0, \Omega)>0$ such that
\be\label{vz-starting bdd}
\| v(t)\|_{L^\frac{2q}{2-q}}+\| z( t)\|_{L^\frac{2q}{2-q}}+\| v( t)\|_{W^{1,q}}+\| z( t)\|_{W^{1,q}}\leq C_3,\ \forall t\in(0, T_m).
\ee
 \end{lemma}
 \begin{proof}In the case of $\tau_1, \tau_2>0$, using the widely known smoothing $L^p$-$L^q$ estimates of the Neumann heat semigroup $\{e^{t\Delta}\}_{t\geq0}$ in $\Omega$, see, e.g. \cite{Win10-JDE, Cao15} and applying those  estimates to the second and  fourth equation in \eqref{Loop equations}, we can readily deduce \eqref{gradvz-bdd-lp}, cf. \cite{A21, KS08, Xiang15-JDE}. In the case of $\tau_1=\tau_2=0$, the standard well-known $W^{2,p}$-elliptic theory  (see e.g. \cite{ADN59, ADN64,Ladyzenskaja710}) easily lead to \eqref{gradvz-bdd-lp}. Because of the  conservations of $u$ and $w$ in \eqref{uvwz-l1}, we first take $p=1$ in \eqref{q-exp},  and then from \eqref{gradvz-bdd-lp} and Sobolev embedding, we arrive at the desired estimate \eqref{vz-starting bdd}.
 \end{proof}
 Based on Sobolev and H\"{o}lder  inequalities, the following Poincar\'{e}-type  inequality follows. In convex domains, upon an observation of the optimal constant for Poincar\'{e} inequality \cite{PW60}, we find an explicit lower bound for the involving constant.
\begin{lemma}\label{k-const}
Let $\Omega\subset \mathbb{R}^2$ be a  bounded and smooth domain and let $k$ be defined in \eqref{k-def}. Then for any a.e nonnegative function $\varphi\in W^{1,2}(\Omega)$ with  $\bar{\varphi}=1$, one has
\be\label{Poin-const}
\|\varphi-1\|^2_{L^2}\leq k\|\nabla \varphi^{\frac{1}{2}}\|^2_{L^2}.
\ee
If,  furthermore,   $\Omega$ is convex, then  $k\geq\frac{4d^2}{\pi^2}$ with $d$ being the diameter of $\Omega$.
\end{lemma}
\begin{proof} The validity of \eqref{Poin-const} is proven in   \cite[Lemma 2.3]{GZ98}. We here re-show the simple proof with emphasis on the explicit lower bound of $k$ in convex domains. In such cases, it follows from \cite{PW60} that the optimal constant for the Poincar\'{e} inequality
\begin{equation}\label{Poincare inequlity1}
\|\varphi-1\|^2_{L^2}\leq \mu_1\|\nabla \varphi\|^2_{L^2}
\end{equation}
is given by
$$
\mu_1= \frac{d^2}{\pi^2}.
$$
The 2D Sobolev embedding $W^{1,1}(\Omega)\hookrightarrow L^2(\Omega)$ implies  there exists $\mu_2>0$ such that
\begin{equation}\label{Poincare inequlity2}
\begin{split}
\|\varphi-1\|^2_{L^2}\leq \mu_2\|\nabla \varphi\|^2_{L^1}.
\end{split}
\end{equation}
This shows that $k$ defined in \eqref{k-def} makes sense,  and,  the optimal constant of $\mu_2$ is
\be\label{mu-opt}
\mu_2=\frac{k}{4|\Omega|}.
\ee
Thus, by  H\"{o}lder  inequality and  the optimality of $\mu_1$ in \eqref{Poincare inequlity1}, we have
\begin{equation*}
\begin{split}
\mu_2|\Omega|\geq \mu_1 =\frac{d^2}{\pi^2}.
\end{split}
\end{equation*}
  By \eqref{Poincare inequlity2} with \eqref{mu-opt},  H\"{o}lder inequality and the fact that $\|\varphi\|_{L^1}=|\Omega|$, we deduce
\begin{equation*}
\begin{split}
\|\varphi-1\|^2_{L^2}\leq \mu_2\|\nabla \varphi\|^2_{L^1}&=4\mu_2\|\varphi^{\frac{1}{2}}\cdot\nabla\varphi^{\frac{1}{2}}\|^2_{L^1}\\[0.2cm]
&\leq4\mu_2|\Omega|\cdot\|\nabla \varphi^{\frac{1}{2}}\|^2_{L^2}=k\|\nabla \varphi^{\frac{1}{2}}\|^2_{L^2},
\end{split}
\end{equation*}
which readily  shows  \eqref{Poin-const}  with $k\geq \frac{4d^2}{\pi^2}$, as desired.
\end{proof}
For convenience of reference,  we state the following version of Trudinger-Moser inequality  and the widely-known  2D Gagliardo-Nirenberg interpolation inequality.
\begin{lemma}[Trudinger-Moser inequality  \cite{ A12}]\label{TM-ineq} Let $\Omega\subset \mathbb{R}^2$ be a bounded and smooth domain. Define
 \begin{equation*}
\begin{split}
\pi^*=\left\{
\begin{array}{llll}
8\pi, &\text{  if    } \Omega=B_R(0),\\
4\pi, &\text{otherwise}.
\end{array}
\right.
\end{split}
\end{equation*}
  Then, for any $\epsilon>0$, there exists $C_\epsilon=C(\epsilon,\Omega)>0$ such that
$$
\int_{\Omega}\exp{|f|}\leq C_\epsilon\exp\left\{(\frac{1}{2\pi^*}+\epsilon)\|\nabla f\|^2_{L^2(\Omega)}+\frac{2}{|\Omega|}\|f\|_{L^1(\Omega)}\right\}, \ \ \forall  f\in H^1(\Omega).
 $$
 \end{lemma}

 \begin{lemma}[2D Gagliardo-Nirenberg interpolation inequality \cite{Fri-bk,JX18-DCDSB, LL16}]\label{G-N}
Let $\Omega\subset \mathbb{R}^2$ be a bounded  and  smooth domain. Let $i$ and $j$ be any integers satisfying $0\leq i<j$, and let $0<p, q, r\leq \infty$, and $\frac{i}{j}\leq \theta\leq 1$ such that
$$
\frac{1}{p}-\frac{i}{2}=\theta\left(\frac{1}{q}-\frac{j}{2}\right)+(1-\theta)\frac{1}{r}.
$$
Then  there exists  $C>0$ depending only on $\Omega, p, q,r, i$ and $j$ such that
\begin{equation}\label{G-N-I}
\|D^i f\|_{L^p}\leq C\left(\|D^j f\|_{L^q}^\theta\|f\|_{L^r}^{1-\theta}+\|f\|_{L^r}\right), \quad \forall f\in W^{j,q}
\end{equation}
with the following exception: if $1<q<\infty$ and $j-i-\frac{2}{q}$ is a nonnegative integer, then \eqref{G-N-I}  holds only  for $\theta$ satisfying $\frac{i}{j}\leq \theta<1$.
\end{lemma}

\section{Product of masses on boundedness vs blow-up}
In the section, we shall prove  boundedness,  gradient estimates and blow-up as stated in (B1), (B2) and (B4) of classical solutions to the IBVP \eqref{Loop equations}. We divide   this section into four subsections to make the flow of our ideas more smooth.
\subsection{From $L^1$ to $L^2$} We start with the following energy identity, which plays a crucial role for our purpose, especially when $\tau_1+\tau_2=0$ or $\chi_2=0$.
\begin{lemma}\label{Lyapunov-f-lemma}The local-in-time classical solution $(u,v,w,z)$ of \eqref{Loop equations} fulfills
\be\label{Lyapunov-f}
\begin{split}
\mathcal{F}^\prime(t)=&-(\tau_1+\tau_2)\chi_1\chi_2\int_\Omega v_tz_t-\tau_1\chi_1\chi_3\int_\Omega v_t^2-\chi_2\int_\Omega u|\nabla(\ln u-\chi_1v)|^2\\
&\ \ -\chi_1\int_\Omega w|\nabla(\ln w-\chi_2z-\chi_3v)|^2,\ \ t\in(0, T_m),
\end{split}
\ee
where $\mathcal{F}(t)$ is defined by
\be\label{F-def}
\begin{split}
\mathcal{F}(t)=&\chi_2\int_\Omega u\ln u+\chi_1\int_\Omega w\ln w-\chi_1\chi_2\int_\Omega \left(uv+wz\right)-\chi_1\chi_3\int_\Omega wv\\
&+\chi_1\chi_2\int_\Omega \left(vz+ \nabla v \cdot\nabla z\right)+\frac{\chi_1\chi_3}{2}\int_\Omega \left(v^2+ |\nabla v|^2\right),\ \ t\in(0, T_m).
\end{split}
\ee
\end{lemma}
\begin{proof}Multiplying the first equation in \eqref{Loop equations}  by $\ln u-\chi_1v$ and integrating by parts over $\Omega$, we obtain upon noticing $\int_\Omega u_t=0$ that
\be\label{u-id1}
\begin{split}
-\int_\Omega u|\nabla(\ln u-\chi_1v)|^2&=\int_\Omega u_t(\ln u-\chi_1v)\\
&=\frac{d}{dt}\int_\Omega\left(u\ln u-\chi_1uv\right)+\chi_1\int_\Omega uv_t.
 \end{split}
\ee
Similarly, multiplying the third equation in \eqref{Loop equations} by $\ln w-\chi_2z-\chi_3v$, we see that
\be\label{w-id1}
\begin{split}
&-\int_\Omega w|\nabla(\ln w-\chi_2z-\chi_3v)|^2\\
&=\frac{d}{dt}\int_\Omega\left(w\ln w-\chi_2wz-\chi_3wv\right)+\chi_2\int_\Omega wz_t+\chi_3\int_\Omega wv_t.
\end{split}
\ee
Next, from the facts that $u=\tau_2z_t-\Delta z+z$ and $w=\tau_1v_t-\Delta v+v$ due to \eqref{Loop equations}, we deduce from integration by parts that
\be\label{mixed-id1}
\begin{split}
\int_\Omega uv_t+\int_\Omega wz_t&=\int_\Omega \left(\tau_2z_t-\Delta z+z\right)v_t+\int_\Omega (\tau_1v_t-\Delta v+v)z_t\\
&=(\tau_1+\tau_2)\int_\Omega v_tz_t+\frac{d}{dt}\int_\Omega\nabla v\cdot\nabla z+\frac{d}{dt}\int_\Omega vz
 \end{split}
\ee
and that
\be\label{mixed-id2}
\int_\Omega wv_t=\int_\Omega (\tau_1v_t-\Delta v+v)v_t=\tau_1\int_\Omega v_t^2+\frac{1}{2}\frac{d}{dt}\int_\Omega\left(v^2+|\nabla v|^2\right).
\ee
Finally,  multiplying \eqref{u-id1} by $\chi_2$ and \eqref{w-id1} by $\chi_1$ and then using \eqref{mixed-id1} and \eqref{mixed-id2}, we can readily end up with \eqref{Lyapunov-f} with $\mathcal{F}$ given by \eqref{F-def}.
\end{proof}

\begin{lemma}\label{u+w-log-bdd-lem} When  $\tau_1= \tau_2=0$, with $\pi^*$ defined in Lemma  \ref{TM-ineq}, assume that
\be\label{key-cond}
\begin{split}
  m_1m_2\chi_1\chi_2< \left(\pi^*-m_2\chi_3\right)\pi^*.
\end{split}
\ee
Then there exists $C=C(u_0, w_0, \Omega)>0$ such that
\be\label{u+w-log-bdd}
\|(u\ln u)(t)\|_{L^1}+\|(w\ln w)(t)\|_{L^1}+\|v(t)\|_{H^1}+\|z(t)\|_{H^1}\leq C, \ \ \ \forall t\in(0, T_m).
\ee
\end{lemma}
\begin{proof}
For our later purpose, thanks to \eqref{key-cond}, we first pick positive constants $a, b$ and $\epsilon$ according to
\be\label{ab-def}
a=\frac{\pi^*}{m_1}, \ \ b=\frac{\pi^*}{m_2}, \ \ \epsilon=\frac{(\pi^*-m_2\chi_3)\pi^*-m_1m_2\chi_1\chi_2}{6\left(\pi^*\right)^3},
\ee
and then,  by direct but tedious computations, we find that
\be\label{AB-def}
\begin{cases}
A:=[(1-2\pi^*\epsilon)\frac{\pi^*}{2m_2}-\frac{\chi_3}{2}]\chi_1>0, \ \ B:=\left(1-2\pi^*\epsilon\right)\frac{\pi^*\chi_2}{2m_1}>0,\\[0.25cm]
 (A-\frac{\chi_1^2\chi_2^2}{4B})>0, \ \ \ (B-\frac{\chi_1^2\chi_2^2}{4A})>0.
 \end{cases}
\ee
Since  $(\tau_1,\tau_2)=(0,0)$, we first infer from \eqref{Loop equations} that
\be\label{ids-int}
\begin{cases}
  \int_\Omega |\nabla z|^2+\int_\Omega z^2=\int_\Omega uz,\ \ \int_\Omega |\nabla v|^2+\int_\Omega v^2=\int_\Omega wv,\\[0.25cm]
  \int_\Omega \nabla v\cdot\nabla z+\int_\Omega vz=\int_\Omega wz,\ \ \int_\Omega \nabla v\cdot\nabla z+\int_\Omega vz=\int_\Omega uv;
  \end{cases}
\ee
then, by the definition of $\mathcal{F}$ in  \eqref{F-def} and the choices of $a,b$ in \eqref{ab-def}, we deduce   that
\be\label{V-sim-def}
\begin{split}
\mathcal{F}(t)& =\chi_2\int_\Omega u\ln u+\chi_1\int_\Omega w\ln w-\chi_1\chi_2\int_\Omega wz-\frac{\chi_1\chi_3}{2}\int_\Omega \left(v^2+ |\nabla v|^2\right)\\
&\ \ =\chi_2\int_\Omega \left(u\ln u-a uz\right)+\chi_1\int_\Omega \left(w\ln w-b wv\right)\\
&\ \ +a\chi_2\int_\Omega uz+b\chi_1\int_\Omega wv-\chi_1\chi_2\int_\Omega wz-\frac{\chi_1\chi_3}{2}\int_\Omega \left(v^2+ |\nabla v|^2\right)\\
&\ \ =-\chi_2\int_\Omega u\ln \frac{e^{az}}{u}-\chi_1\int_\Omega w\ln \frac{e^{bv}}{w}+\left(b-\frac{\chi_3}{2}\right)\chi_1\int_\Omega \left(v^2+ |\nabla v|^2\right)\\
&\ \ \ \  -\chi_1\chi_2\int_\Omega \left(vz+\nabla v\cdot\nabla z\right)+a\chi_2\int_\Omega\left(z^2+|\nabla z|^2\right).
\end{split}
\ee
 Observe that $-\ln s$ is a convex function in $s$, $\int_\Omega\frac{u}{m_1}=1$ and $\int_\Omega\frac{w}{m_2}=1$ due to mass conservations of $u$ and $w$. Then Jensen's inequality tells us that
\begin{equation*}{\label{4.14}}
\begin{split}
-\ln \left(\frac{1}{m_1}\int_\Omega e^{az}\right)=
-\ln \int_\Omega\frac{e^{az}}{u}\frac{u}{m_1}\leq
-\frac{1}{m_1}\int_\Omega u\ln \frac{e^{a z}}{u}
\end{split}
\end{equation*}
and, similarly,
\begin{equation*}{\label{4.14}}
\begin{split}
-\ln \left(\frac{1}{m_2}\int_\Omega e^{bv}\right)\leq
-\frac{1}{m_2}\int_\Omega w\ln \frac{e^{bv}}{w}.
\end{split}
\end{equation*}
Now,   for the specifications  of $a,b, \epsilon$ in \eqref{ab-def}, we apply the Trudinger-Morser inequality in Lemma \ref{TM-ineq}  along  with the  boundedness  $\| z\|_{L^1}=\| u_0\|_{L^1}$ and $\| v\|_{L^1}=\| w_0\|_{L^1}$  due to $(\tau_1,\tau_2)=(0,0)$ to find a  $C=C(\Omega)>0$ such that
\be\label{TM-involves1}
\begin{split}
&-\chi_2\int_\Omega u\ln \frac{e^{az}}{u}-\chi_1\int_\Omega w\ln \frac{e^{bv}}{w}\\
&\geq -m_1\chi_2\ln \left(\frac{1}{m_1}\int_\Omega e^{az}\right)-m_2\chi_1\ln \left(\frac{1}{m_2}\int_\Omega e^{bv}\right)\\
&\geq -m_1\chi_2\left[\ln \frac{C}{m_1}+ (\frac{1}{2\pi^*}+\epsilon)a^2\|\nabla z\|^2_{L^2}+\frac{2a}{|\Omega|}\| z\|_{L^1}\right]\\
&\ \  -m_2\chi_1\left[\ln \frac{C}{m_2}+ (\frac{1}{2\pi^*}+\epsilon)b^2\|\nabla v\|^2_{L^2}+\frac{2b}{|\Omega|}\| v\|_{L^1}\right]\\
&\geq - m_1\chi_2a^2(\frac{1}{2\pi^*}+\epsilon)\|\nabla z\|^2_{L^2}-m_2\chi_1b^2(\frac{1}{2\pi^*}+\epsilon)\|\nabla v\|^2_{L^2}-D,
\end{split}
\ee
where $D$ is a finite number and is defined by
$$
D=m_1\chi_2\left[\ln \frac{C}{m_1}+\frac{2a}{|\Omega|}\| u_0\|_{L^1}\right]+m_2\chi_1\left[\ln \frac{C}{m_2}+\frac{2b}{|\Omega|}\| w_0\|_{L^1}\right].
$$
Substituting \eqref{TM-involves1} into \eqref{V-sim-def} and using \eqref{AB-def}, we conclude that
\be\label{F-fin-sub}
\begin{split}
\mathcal{F}(t)&-\int_\Omega\left[\left(\frac{\pi^*}{m_2}-\frac{\chi_3}{2}\right)\chi_1 v^2-\chi_1\chi_2vz+\frac{\pi^*\chi_2}{m_1}z^2\right]+D\\
&\geq \int_\Omega\left(A|\nabla v|^2-\chi_1\chi_2 \nabla v\cdot\nabla z+ B|\nabla z|^2\right)\\
&=\int_\Omega\left(\sqrt{A}\nabla v-\frac{\chi_1\chi_2}{2\sqrt{A}} \nabla z\right)^2+\left(B-\frac{\chi_1^2\chi_2^2}{4A}\right)\int_\Omega |\nabla z|^2\\
&=\int_\Omega\left(\sqrt{B}\nabla z-\frac{\chi_1\chi_2}{2\sqrt{B}} \nabla v\right)^2+\left(A-\frac{\chi_1^2\chi_2^2}{4B}\right)\int_\Omega |\nabla v|^2.
\end{split}
\ee
By the 2D G-N inequality in \eqref{G-N-I},  for any $\eta>0$, there exists $C_\eta>0$ such that
$$
\|\phi\|_{L^2}^2\leq \eta\|\nabla \phi\|_{L^2}^2+C_\eta\|\phi\|_{L^1}^2, \  \  \ \  \forall \phi\in H^1.
$$
Combining this with \eqref{F-fin-sub}, \eqref{AB-def} and the decreasing monotonicity of $\mathcal{F}$ implied by \eqref{Lyapunov-f} with  $(\tau_1, \tau_2)=(0,0)$, we infer that there exists a constant $E>0$ such that
\be\label{vz-grd-est}
 \|v\|_{H^1}^2+\|z\|_{H^1}^2\leq E \mathcal{F}(t)+E\leq E\mathcal{F}(0)+E,
\ee
which,  along with   \eqref{V-sim-def}, \eqref{ids-int}  and the fact that $s\ln s\geq -e^{-1}$ for $s>0$,  further enables us to deduce that
\be\label{uw-log-est}
\begin{split}
&\chi_2\int_\Omega |u\ln u|+\chi_1\int_\Omega |w\ln w|\\
&=\chi_2\left(\int_\Omega u\ln u -2\int_{\{u\leq 1\}} u\ln u\right)+\chi_1\left(\int_\Omega w\ln w -2\int_{\{w\leq 1\}} w\ln w\right)\\
&\leq \chi_2\int_\Omega u\ln u+\chi_1\int_\Omega w\ln w+2(\chi_1+\chi_2)e^{-1}|\Omega|\\
&=\mathcal{F}(t)+\chi_1\chi_2\int_\Omega \left(vz+\nabla v\cdot\nabla z\right)+\frac{\chi_1\chi_3}{2}\|v\|_{H^1}^2+2(\chi_1+\chi_2)e^{-1}|\Omega|\\
&\leq \mathcal{F}(0)+(\chi_2+|\chi_3|)\chi_1\left(\|v\|_{H^1}^2+\|z\|_{H^1}^2\right)+2(\chi_1+\chi_2)e^{-1}|\Omega|. \end{split}
\ee
 Consequently, our desired estimate \eqref{u+w-log-bdd} follows readily from \eqref{vz-grd-est} and \eqref{uw-log-est}.
\end{proof}
\begin{remark} By simpler arguments, when $\chi_1\chi_2=0$ and $m_2\chi_3<\pi^*$, no matter whether $(\tau_1,\tau_2)=(0,0)$ or not, one  can easily show that \eqref{u+w-log-bdd} is still valid.
\end{remark}
In the fully parabolic case, we shall derive an analog of Lemma \ref{u+w-log-bdd-lem}  under an implicit smallness condition on the product of  masses of $u$ and $w$.

\begin{lemma}\label{fully PP}
In the fully parabolic case, i.e., $\tau_1>0, \tau_2>0$, assume that
 \begin{equation}\label{small mass}
  4m_1m_2\chi_1\chi_2C_{GN}^8<\sqrt{1-4m_2\chi_3C_{GN}^4},
\end{equation}
 where the Gagilardo-Nirenberg  constant $C_{GN}=C_{GN} (\Omega)>0$ is determined by \eqref{GN-inequ}. Then there exists  $C=C(u_0, v_0,  w_0, z_0, \tau_i, \chi_i, \Omega)>0$ such that
\be\label{u+w-log-bdd-fp}
\|(u\ln u)(t)\|_{L^1}+\|(w\ln w)(t)\|_{L^1}+\|v(t)\|_{H^1}+\|z(t)\|_{H^1}\leq C, \ \ \ \forall t\in(0, T_m).
\ee
\end{lemma}
\begin{proof}
Multiplying the first and the third equation  in \eqref{Loop equations} by $\ln u$ and $\ln w$ ,  and then,  multiplying the second and the fourth equation by $-\Delta v$ and $-\Delta z$, respectively, and finally  integrating over $\Omega$ by parts, we compute, for $t\in(0, T_m)$, that
\be\label{id-ulnu+wlnw}
\begin{cases}
&\frac{d}{dt}\int_{\Omega}u\ln u+4\int_\Omega|\nabla u^{\frac{1}{2}}|^2=-\chi_1\int_{\Omega}u\Delta v, \\[0.25cm]
&\frac{d}{dt}\int_\Omega w\ln w+4\int_\Omega|\nabla w^{\frac{1}{2}}|^2=-\chi_2\int_\Omega w\Delta z-\chi_3\int_\Omega w\Delta v, \\[0.25cm]
&\frac{\tau_1}{2}\frac{d}{dt}\int_\Omega|\nabla v|^2+\int_\Omega|\Delta v|^2+\int_{\Omega}|\nabla v|^2=-\int_\Omega w\Delta v, \\[0.25cm]
&\frac{\tau_2}{2}\frac{d}{dt}\int_\Omega|\nabla z|^2+\int_\Omega|\Delta z|^2+\int_{\Omega}|\nabla z|^2=-\int_\Omega u\Delta z.
\end{cases}
\ee
Given any positive constants $a,b$ and $c$, to be specified below as in \eqref{abc-choice}, through an elementary linear combination of \eqref{id-ulnu+wlnw}, we arrive at
\be\label{3.12}
\begin{split}
&\frac{d}{dt}\int_\Omega\left(u\ln u+aw\ln w+\frac{b\tau_1}{2}|\nabla v|^2+\frac{c\tau_2}{2}|\nabla z|^2\right)+4\int_\Omega|\nabla u^{\frac{1}{2}}|^2\\
&+4a\int_\Omega|\nabla w^{\frac{1}{2}}|^2+b\int_\Omega|\Delta v|^2+b\int_\Omega|\nabla v|^2+c\int_\Omega|\Delta z|^2+c\int_\Omega|\nabla z|^2\\
&=-\int_\Omega \left(\chi_1u+a\chi_3w+bw\right)\Delta v-\int_\Omega \left(a\chi_2w+cu\right)\Delta z.
\end{split}
\ee
Using basic Cauchy-Schwarz inequality, we estimate the right-hand side as
\be\label{3.13}
\begin{split}
&-\int_\Omega \left(\chi_1u+a\chi_3w+bw\right)\Delta v-\int_\Omega \left(a\chi_2w+cu\right)\Delta z\\
&\leq b\int_\Omega |\Delta v|^2+\frac{1}{4b}\int_\Omega\left(\chi_1u+a\chi_3w+bw\right)^2\\
&\ \ \ +c\int_\Omega|\Delta z|^2+\frac{1}{4c}\int_\Omega \left(a\chi_2w+cu\right)^2\\
&\leq b\int_\Omega |\Delta v|^2+\frac{1}{2b}\int_\Omega\left[\chi^2_1u^2+(a\chi_3+b)^2w^2\right]\\
&\ \ +c\int_\Omega|\Delta z|^2+\frac{1}{2c}\int_\Omega\left(a^2\chi^2_2w^2+c^2u^2\right)\\
&=b\int_\Omega |\Delta v|^2+c\int_\Omega|\Delta z|^2 +\left(\frac{\chi^2_1}{2b}+\frac{c}{2}\right)\int_\Omega u^2\\
 &\ \ \  +\left(\frac{(a\chi_3+b)^2}{2b}+\frac{a^2\chi^2_2}{2c}\right)\int_\Omega w^2, \quad \ \ \forall t\in(0, T_m).
\end{split}
\ee
By the  2D Gaglarido-Nirenberg  interpolation inequality in Lemma \ref{G-N} and the elementary  fact that  $(X+Y)^4\leq 2^3(X^4+Y^4)$ for all $ X,Y\geq 0$, we infer there exists a constant $C_{GN}=C_{GN} (\Omega)>0$ such that
\be\label{GN-inequ}
\begin{split}
\int_\Omega \phi^2=\|\phi^\frac{1}{2}\|_{L^4}^4&\leq C_{GN}^4 \left(\|\nabla \phi^\frac{1}{2}\|_{L^2}^\frac{1}{2}\|\phi^\frac{1}{2}\|_{L^2}^\frac{1}{2}
+\|\phi^\frac{1}{2}\|_{L^2}\right)^4\\
&\leq 8C_{GN}^4 \|\phi\|_{L^1}\|\nabla \phi^\frac{1}{2}\|_{L^2}^2+  8C_{GN}^4\|\phi\|_{L^1}^2, \ \ \  \forall \phi^\frac{1}{2}\in W^{1,2}.
\end{split}
\ee
Recalling that $\|u\|_{L^1}=\|u_0\|_{L^1}=m_1$ and $\|w\|_{L^1}=\|w_0\|_{L^1}=m_2$ by \eqref{uvwz-l1}, we employ \eqref{GN-inequ} twice to finally estimate \eqref{3.13} as
\be\label{3.14}
\begin{split}
&-\int_\Omega \left(\chi_1u+a\chi_3w+bw\right)\Delta v-\int_\Omega \left(a\chi_2w+cu\right)\Delta z\\
&\leq b\int_\Omega |\Delta v|^2+c\int_\Omega|\Delta z|^2+4m_1\left(\frac{\chi^2_1}{b}+c\right)C_{GN}^4\int_\Omega |\nabla u^\frac{1}{2}|^2\\
&\ \ +4m_2\left(\frac{(a\chi_3+b)^2}{b}+\frac{a^2\chi^2_2}{c}\right)C_{GN}^4\int_\Omega |\nabla w^\frac{1}{2}|^2+C_1, \ \ \forall t\in(0, T_m).
\end{split}
\ee
where $C_1$ is a finite number given by
$$
C_1=4m_1^2\left(\frac{\chi^2_1}{b}+c\right)C_{GN}^4
+4m_2^2\left(\frac{(a\chi_3+b)^2}{b}+\frac{a^2\chi^2_2}{c}\right)C_{GN}^4.
$$
Finally, substituting \eqref{3.14} into \eqref{3.12}, we end up with a key ODI as follows:
\be\label{u+w-key odi}
\begin{split}
&\frac{d}{dt}\int_\Omega\left(u\ln u+aw\ln w+\frac{b\tau_1}{2}|\nabla v|^2+\frac{c\tau_2}{2}|\nabla z|^2\right)\\
&\ \ +4A\int_\Omega|\nabla u^{\frac{1}{2}}|^2+4B\int_\Omega|\nabla w^{\frac{1}{2}}|^2+b\int_\Omega|\nabla v|^2+c\int_\Omega|\nabla z|^2\\
&\leq C_1, \quad \ \ \forall t\in(0, T_m).
\end{split}
\ee
where the constants $A$ and $B$ are given by
\be\label{AB-choices-bdd}
\begin{cases}
&A=1-m_1\left(\frac{\chi^2_1}{b}+c\right)C_{GN}^4:=p^{-1}\left(p-\frac{\chi^2_1}{b}-c\right), \\[0.2cm]
&B=a-m_2\left(\frac{(a\chi_3+b)^2}{b}+\frac{a^2\chi^2_2}{c}\right)C_{GN}^4:=q^{-1}\left(aq
-\frac{(a\chi_3+b)^2}{b}
-\frac{a^2\chi^2_2}{c}\right).
\end{cases}
\ee
To gain something out of \eqref{u+w-key odi}, we wish that both $A$ and $B$ be positive, which is possible only when $ 16\chi_1^2\chi_2^2<p^2q(q-4\chi_3)$, equivalent to our assumption \eqref{small mass}. In such case, we can specify, for instance,  positive constants  $a,b$ and $c$ as
\be\label{abc-choice}
\begin{cases}
&b=\frac{pq(q-4\chi_3)}{8\chi_2^2}>0,\\[0.2cm]
&a=\frac{(bp-\chi_1^2)(q-2\chi_3)b}{2\left[b^2\chi_2^2+(bp-\chi_1^2)\chi_3^2\right]}>0,\ \ \ \ \  c=\frac{(bp-\chi_1^2)}{2b}+\frac{a^2b\chi_2^2}{2\left[a b q-(a\chi_3+b)^2\right]}>0
\end{cases}
\ee
so that $A$ and $B$ defined in \eqref{AB-choices-bdd} are positive.  Next, notice, for any $\epsilon>0$, one has that $s \ln s\leq \epsilon s^2+C_\epsilon$ with finite $C_\epsilon=\sup\{ s\ln s-\epsilon s^2: s>0\}$. Therefore, one can readily deduce from \eqref{GN-inequ}, for some $C_2, C_3>0$,  that
$$
\int_\Omega u\ln u\leq A\int_\Omega |\nabla u^\frac{1}{2}|^2+C_2, \ \ \ \ a\int_\Omega w\ln w\leq B\int_\Omega |\nabla w^\frac{1}{2}|^2+C_3.
$$
Combining this with \eqref{u+w-key odi}, we finally find a positive $C_4>0$ such that
\begin{equation*}
\begin{split}
&\frac{d}{dt}\int_\Omega\left(u\ln u+aw\ln w+\frac{b\tau_1}{2}|\nabla v|^2+\frac{c\tau_2}{2}|\nabla z|^2\right)\\
&\ \ +\min\left\{1, \ \frac{2}{\tau_1}, \ \frac{2}{\tau_2}\right\}\int_\Omega\left(u\ln u+aw\ln w+\frac{b\tau_1}{2}|\nabla v|^2+\frac{c\tau_2}{2}|\nabla z|^2\right)\\
&\ \ \leq C_4,  \quad \ \ \forall t\in(0, T_m).
\end{split}
\end{equation*}
Solving this simple  differential inequality and using the simple trick used in \eqref{uw-log-est}, we find a  positive $C_5>0$ such that
$$
\|u\ln u\|_{L^1}+\|w\ln w\|_{L^1}+\|\nabla v\|_{L^2}+\|\nabla z\|_{L^2}\leq C_5, \quad \ \ \forall t\in(0, T_m),
$$
which along with  \eqref{vz-starting bdd} with $q=1$ yields our desired estimate \eqref{u+w-log-bdd-fp}.
\end{proof}
Armed with the key uniform boundedness of $(u\ln u, w\ln w)$ as obtained in Lemmas \ref{u+w-log-bdd-lem} and \ref{fully PP},  it is quite standard for us to show higher  $L^p$-boundedness ($p>1$) and, eventually,  $L^\infty$-boundedness as in \eqref{bdd-thm-fin} in 2D setting, see similar situations in \cite{A4, Xiang15-JDE, Xiang18-NA, Xiang18-JMP}. We here would like to supply an argument for  \eqref{Loop equations} for the sake of completeness and for clarity of deriving higher order gradients in Subsection 3.3.

\begin{lemma}\label{uw-l2-bdd-lem} When  $\tau_1= \tau_2=0$, assume that  \eqref{key-cond} holds. Then there exists a constant $C=C(u_0, w_0, \Omega)>0$ such that
\be\label{u+w-l2-bdd}
\|u(t)\|_{L^2}+\|w(t)\|_{L^2}+\|v(t)\|_{H^2}+\|z(t)\|_{H^2}\leq C,\ \  \  \forall t\in(0, T_m);
\ee
and, for any $q\in(1, \infty)$,    there exists $C_q=C(q,u_0, w_0, \Omega)>0$ such that
\be\label{v+z-gradlq-bdd}
\|v(t)\|_{W^{1,q}}+\|z(t)\|_{W^{1,q}}\leq C_q, \ \ \   \forall t\in(0, T_m).
\ee
\end{lemma}

\begin{proof}
Applying the elliptic estimate in \cite[Lemma 2.7]{JX18-DCDSB} to  the second and fourth equation  with $\tau_1=\tau_2=0$ in \eqref{Loop equations}, we see, for any $\epsilon>0$ and $p>1$, there  exists a positive constant $C_\epsilon>0$ such that
\be \label{vz-lp-by-uw}
 \int_\Omega\left( v^p, \ \  w^p\right)\leq \epsilon \int_\Omega \left(w^p,\ \ u^p\right)+C_\epsilon.
\ee
Using the  equations in \eqref{Loop equations}  with $\tau_1=\tau_2=0$, performing  integration by parts and using Young's inequality and \eqref{vz-lp-by-uw}, we compute that
 \be \label{uw-l2-diff}
 \begin{split}
&3\frac{d}{dt} \int_\Omega \left( u^2+w^2\right)+ 6\int_\Omega |\nabla u |^2+6\int_\Omega |\nabla w |^2\\
&=3\chi_1\int_\Omega u^2(w-v)+3\chi_2\int_\Omega w^2(u-z)+3\chi_3\int_\Omega w^2(w-v)\\
&\leq \left(4\chi_1+\chi_2\right)\int_\Omega u^3+\left(\chi_1+|\chi_3|\right)\int_\Omega v^3\\
&\  +\left(\chi_1+4\chi_2+5|\chi_3|\right)\int_\Omega w^3+\chi_2\int_\Omega z^3\\
&\leq \left(4\chi_1+4\chi_2+6|\chi_3|\right)\int_\Omega \left(u^3+w^3\right)+C_1, \ \ t\in (0, T_m).
\end{split}
\ee
Due to the uniform $L^1$-boundedness of $(u\ln u, w\ln w)$ in \eqref{u+w-log-bdd}, the 2D G-N inequality involving logarithmic functions  from  \cite[Lemma A.5]{TW14-JDE} implies, for any $\eta>0$, there exists $C_\eta>0$ such that
\be\label{ul3-bdd by}
\int_\Omega \left(u^3, \ \ w^3\right)\leq \eta \int_\Omega \left( |\nabla u |^2, \ \ |\nabla w|^2\right)+C_\eta.
\ee
Of course, the above inequality or the usual 2D G-N inequality simply shows, for any $\sigma>0$, there exists $C_\sigma>0$ such that
\be\label{ul2-bdd by}
\int_\Omega \left(u^2, \ \ w^2\right)\leq \sigma \int_\Omega \left( |\nabla u |^2, \ \ |\nabla w|^2\right)+C_\sigma.
\ee
Based on \eqref{ul2-bdd by}, \eqref{ul3-bdd by} and \eqref{uw-l2-diff}, one can readily derive an ODI of  the form
$$
\frac{d}{dt} \int_\Omega \left( u^2+w^2\right)+\int_\Omega \left( u^2+w^2\right)\leq C(u_0, w_0, \Omega), \quad \quad \forall t\in(0, T_m),
$$
yielding directly the uniform boundedness of $\|u\|_{L^2}+\|w\|_{L^2}$ and the $H^2$-boundedness of $(v, z)$ by the $H^2$-elliptic estimate (cf. \cite{Ladyzenskaja710})  in \eqref{u+w-l2-bdd}. Finally, the $W^{1,q}$-estimate of $(v,z)$ in \eqref{v+z-gradlq-bdd} follows from Lemma \ref{reciprocal-lem} with $p=2$.
\end{proof}
\begin{lemma}\label{uw-l2-bdd-lem-pp} In the fully parabolic case, i.e., $\tau_1>0, \tau_2>0$, assume that \eqref{small mass} holds. Then there exists a constant $C=C(u_0, w_0, \Omega)>0$ such that
\be\label{u+w-l2-bdd-pp}
\|u(t)\|_{L^2}+\|w(t)\|_{L^2}+\|v(t)\|_{H^2}+\|z(t)\|_{H^2}\leq C, \  t\in(\min\{1, \frac{T_m}{2}\}, T_m);
\ee
and, for any $q\in(1, \infty)$,    there exists $C_q=C(q,u_0, w_0, \Omega)>0$ such that
\be\label{v+z-gradlq-bdd-pp}
\|v(t)\|_{W^{1,q}}+\|z(t)\|_{W^{1,q}}\leq C_q, \ \ \ \forall t\in(\min\{1, \frac{T_m}{2}\}, T_m).
\ee
\end{lemma}
\begin{proof}
Using the equations  and homogeneous Neumann boundary conditions in the IBVP \eqref{Loop equations}, we find, upon integration by parts,  for $t\in(0, T_m)$,  that
\be\label{u+w-l2-diff}
\begin{cases}
\frac{d}{dt} \int_\Omega u^2+ 2\int_\Omega |\nabla u |^2 =-\chi_1  \int_\Omega  u^2 \Delta  v,\\[0.25cm]
\frac{d}{dt} \int_\Omega w^2+ 2\int_\Omega |\nabla w |^2 =-\chi_2  \int_\Omega  w^2 \Delta  z-\chi_3  \int_\Omega  w^2 \Delta  v,\\[0.25cm]
\tau_1\frac{d}{dt} \int_\Omega |\Delta v|^2+ 2\int_\Omega |\Delta v |^2+2\int_\Omega |\nabla \Delta v |^2=-2\int_\Omega \nabla w \nabla \Delta v, \\[0.25cm]
\tau_2\frac{d}{dt} \int_\Omega |\Delta z|^2+ 2\int_\Omega |\Delta z |^2+2\int_\Omega |\nabla \Delta z |^2=-2\int_\Omega \nabla u \nabla \Delta z.
\end{cases}
\ee
Adding those identities  in  \eqref{u+w-l2-diff} together, we obtain, for any $\epsilon>0$, that
\be\label{u+w-l2-com}
\begin{split}
&\frac{d}{dt}\int_\Omega\left(u^2+w^2+\tau_1 |\Delta v|^2+ \tau_2  |\Delta z|^2\right)+2\int_\Omega|\nabla u|^2+2\int_\Omega|\nabla w|^2\\
&\ \ +2\int_\Omega|\Delta v|^2+2\int_\Omega|\nabla \Delta v|^2+2\int_\Omega|\Delta z|^2+2\int_\Omega|\nabla \Delta z|^2\\
&=-\chi_1\int_\Omega u^2\Delta v-\chi_2 \int_\Omega w^2\Delta z
-\chi_3 \int_\Omega w^2\Delta v \\
&\quad \ \ -2\int_\Omega \nabla w \nabla \Delta v-2\int_\Omega \nabla u \nabla \Delta z\\
&\leq \left(\chi_1+|\chi_3|\right)\epsilon\int_\Omega|\Delta v|^3+\frac{2\chi_1}{3\sqrt{3\epsilon}}\int_\Omega u^3+\frac{2\left(\chi_2+|\chi_3|\right)}{3\sqrt{3\epsilon}}\int_\Omega w^3\\
&\ \ +\chi_2\epsilon \int_\Omega|\Delta z|^3 +\int_\Omega |\nabla w|^2+\int_\Omega |\nabla \Delta v|^2+\int_\Omega |\nabla u|^2+\int_\Omega |\nabla \Delta z|^2,
\end{split}
\ee
where we have applied the Young's inequality with epsilon a couple of times:
\be\label{Young}
ab\leq \epsilon a^p+\frac{b^q}{(\epsilon p)^{\frac{q}{p}}q},  \ \ p>0, q>0, \frac{1}{p}+\frac{1}{q}=1,    \quad \forall a,b\geq 0.
\ee
Then it is straightforward to see  from \eqref{u+w-l2-com}  that
\be\label{u+w-l2-com-odi}
\begin{split}
&\frac{d}{dt}\int_\Omega\left(u^2+w^2+\tau_1 |\Delta v|^2+ \tau_2  |\Delta z|^2\right)+\int_\Omega|\nabla u|^2+\int_\Omega|\nabla w|^2\\
&\ \ +2\int_\Omega|\Delta v|^2+\int_\Omega|\nabla \Delta v|^2+2\int_\Omega|\Delta z|^2+\int_\Omega|\nabla \Delta z|^2\\
&\leq \left(\chi_1+\chi_2+|\chi_3|\right)\epsilon\int_\Omega\left(|\Delta v|^3+|\Delta z|^3\right)\\
&\ \ \ +\frac{2\left(\chi_1+\chi_2+|\chi_3|\right)}{3\sqrt{3\epsilon}}\int_\Omega \left(u^3+w^3\right), \ \ \ \forall \epsilon>0.
\end{split}
\ee
Applying the 2D G-N interpolation inequality in Lemma \ref{G-N},  Sobolev interpolation inequality and  the boundedness of $\| v\|_{H^1}+\|z\|_{H^1}$ ensured  by \eqref{u+w-log-bdd-fp}, we infer (see details, for instance, in  \cite{Xiang15-JDE}), for some $C_1>0$,  that
 \be\label{delta vl3-bdd by}
\int_\Omega  \left(|\Delta v|^3, \ \ |\Delta z|^3\right)\leq C_1\int_\Omega \left(|\nabla \Delta v|^2, \ \ |\nabla \Delta z|^2\right)+C_1.
\ee
 Based on  \eqref{delta vl3-bdd by}, \eqref{ul3-bdd by} and \eqref{ul2-bdd by}, upon suitably choosing $\epsilon, \eta, \sigma$,  from  \eqref{u+w-l2-com-odi}, we  can easily deduce, for $ t\in(\min\{1, \frac{T_m}{2}\}, T_m)$, a final ODI of the form that
\begin{align*}
&\frac{d}{dt}\int_\Omega\left(u^2+w^2+\tau_1 |\Delta v|^2+ \tau_2  |\Delta z|^2\right)\\
&\  +\min\left\{1,\ \frac{2}{\tau_1},\ \frac{2}{\tau_2}\right\} \int_\Omega\left(u^2+w^2+\tau_1 |\Delta v|^2+ \tau_2  |\Delta z|^2\right)\leq C_2.
\end{align*}
This along  with the standard elliptic  $H^2$-estimate and Lemma \ref{reciprocal-lem} with $p=2$ yields  \eqref{u+w-l2-bdd-pp} and \eqref{v+z-gradlq-bdd-pp}, as wished.
\end{proof}

\subsection{From $L^2$ to $L^\infty$:}
In this subsection, we shall prove the global boundedness claimed  in \eqref{bdd-thm-fin} and thus global existence of solutions to \eqref{Loop equations}.

\begin{lemma} \label{linfty-bdd-lem} Under the conditions of Lemma \ref{u+w-log-bdd-lem} or \ref{fully PP}, the classical solution $(u,v,w,z)$ of \eqref{Loop equations} is global in time and is uniformly bounded according to \eqref{bdd-thm-fin}.
\end{lemma}
\begin{proof}
Multiplying the $u$-equation in \eqref{Loop equations} by $3u^{2}$, integrating over $\Omega$ by parts and applying  the $(L^2, L^8)$-boundedness of $(u, \nabla v)$ in Lemma \ref{uw-l2-bdd-lem} or \ref{uw-l2-bdd-lem-pp},  Young's inequality \eqref{Young} and the 2D G-N inequality, we conclude, for $t\in(0, T_m)$, that
\begin{align*}
\frac{d}{dt}\int_\Omega u^3+\int_\Omega u^3
+3\int_\Omega u |\nabla u|^2&\leq 3\chi_1^2\int_\Omega u^3|\nabla v|^2+\int_\Omega u^3\\
&\leq 3\chi_1^2\int_\Omega   u^4+\frac{3^4\chi_1^2}{4^4}\int_\Omega |\nabla v|^8+\int_\Omega u^3\\
&\leq 4\chi_1^2\| u^\frac{3}{2}\|_{L^\frac{8}{3}}^\frac{8}{3}+C_1\\
&\leq 4\chi_1^2C_2\left(\|\nabla  u^\frac{3}{2}\|_{L^2}^\frac{4}{3}\| u^\frac{3}{2}\|_{L^\frac{4}{3}}^\frac{4}{3}+\| u^\frac{3}{2}\|_{L^\frac{4}{3}}^\frac{8}{3}\right)+C_1\\
&\leq C_3 \|\nabla  u^\frac{3}{2}\|_{L^2}^\frac{4}{3}+C_3\\
&\leq \int_\Omega  u|\nabla u|^2+C_4,
\end{align*}
from which the uniform $L^3$-boundedness of $u$ follows. Applying the same argument to $w$-equation and noticing the $(L^2, L^8, L^8)$-boundedness of $(w, \nabla v, \nabla z)$, one can readily show the uniform $L^3$-boundedness of $w$.  Consequently, a simple application of Lemma \ref{reciprocal-lem} gives rise to
\be\label{u+w-L3-bdd}
\|u\|_{L^3}+\|w\|_{L^3}+\|v\|_{W^{1,\infty}}+\|z\|_{W^{1,\infty}}\leq C_5, \ \ \ \forall t\in(0, T_m).
\ee
To derive the $L^\infty$-boundedness of $u$, based on \eqref{u+w-L3-bdd}, we employ  the variation-of-constants formula for the  $u$-equation in \eqref{Loop equations}   and the well-known smoothing $L^p$-$L^q$-estimates for the Neumann heat  semigroup  $\{e^{t\Delta}\}_{t\geq0}$ (\cite{Win10-JDE, Cao15}) to conclude that
\begin{align*}
\|u(t)\|_{L^\infty}&\leq \|e^{t\Delta }u_0\|_{L^\infty}+\chi_1\int_0^t\left\|e^{(t-s)\Delta }\nabla \cdot((u\nabla v)(s))\right\|_{L^\infty}ds \\
  &\leq \|u_0\|_{L^\infty}+C_6\chi_1\int_0^t\left(1+(t-s)^{-\frac{1}{2}-\frac{1}{3}}\right)e^{-\lambda_1(t-s) }\left\|(u\nabla v)(s)\right\|_{L^3}ds\\
 &\leq  \|u_0\|_{L^\infty}+C_6\chi_1\int_0^t\left(1+(t-s)^{-\frac{1}{2}-\frac{1}{3}}\right)e^{-\lambda_1(t-s) }\left\|u\right\|_{L^3}\left\|\nabla v\right\|_{L^\infty}ds\\
 & \leq \|u_0\|_{L^\infty}+C_7\chi_1\int_0^t\left(1+\sigma^{-\frac{5}{6}}\right)e^{-\lambda_1\sigma }d\sigma\\
   &\leq \|u_0\|_{L^\infty}+C_8\chi_1, \quad   \quad \forall t\in(0, T_m).
  \end{align*}
Here, $\lambda_1(>0)$ is the first nonzero eigenvalue of $-\Delta$ under homogeneous Neumann boundary condition.  Performing the same argument to the variation-of-constants formula for the   $w$-equation  and using \eqref{u+w-L3-bdd}, we get the uniform $L^\infty$-boundedness of $w$. To sum up, we have shown that
\be\label{u+w-Linfty-bdd}
\|u(t)\|_{L^\infty}+\|w(t)\|_{L^\infty}+\|v(t)\|_{W^{1,\infty}}+\|z(t)\|_{W^{1,\infty}}\leq C_9, \ \ \ \forall t\in(0, T_m).
\ee
By the extensibility criterion \eqref{u-w blow-up criterion} in Lemma \ref{2.l1}, we first infer  that $T_m=\infty$, and then,  the desired uniform boundedness \eqref{bdd-thm-fin} is simply  \eqref{u+w-Linfty-bdd}; that is, the classical solution $(u,v,w,z)$ of \eqref{Loop equations} is global in time and is uniformly bounded.
\end{proof}

\subsection{Higher order gradient estimates}

 For our stabilization purpose below, given the uniform $L^1$-boundedness of $(u\ln u, w\ln w)$, we  proceed to show further higher order gradient estimates away from the initial time $t=0$ as stated in \eqref{gradUW-improve-thm}, which is of interest for its own sake, on the other hand.

\begin{lemma}\label{gradUW-improve-lem}  If $\|(u\ln u)(t)\|_{L^1}+\|(w\ln w)(t)\|_{L^1}$ is uniformly bounded on $(0, T_m)$, then $T_m=\infty$ and  the combined higher order gradient estimate \eqref{gradUW-improve-thm} holds.
\end{lemma}
\begin{proof}
Due to the uniform $L^1$-boundedness of $(u\ln u, w\ln w)$, one can use the same arguments as Lemma \eqref{fully PP}, \ref{uw-l2-bdd-lem} or \ref{uw-l2-bdd-lem-pp} to show the uniform boundedness of $\|v\|_{H^2}$ and $\|z(t)\|_{H^2}$  for $t\in (\min\{1, \frac{T_m}{2}\}, T_m)$, and repeating the argument in previous subsections, one can easily obtain first the uniform  estimate \eqref{u+w-Linfty-bdd} with $T_m=\infty$, and then $(u,v,w,z)\in \left(C^{2,1}(\bar{\Omega}\times[1,\infty))\right)^4$. Then, to get higher order gradient estimates, we begin to test the $u$-equation in \eqref{Loop equations} by $-2\Delta u$ and use \eqref{u+w-Linfty-bdd} to get
\be\label{grad U-diff}
\begin{split}
\frac{d}{dt}\int_\Omega |\nabla u|^2&+\int_\Omega |\nabla u|^2+2\int_\Omega |\Delta u|^2\\
&=2\chi_1\int_\Omega \left(\nabla u\nabla v+u\Delta v\right)\Delta u+\int_\Omega |\nabla u|^2\\
&\leq \int_\Omega |\Delta u|^2+C_1\int_\Omega |\nabla u|^2+C_2.
\end{split}
\ee
The 2D Gagilardo-Nirenberg inequality and the $H^2$-elliptic estimate together imply
\be\label{grad-gn} \begin{split}
C_1\|\nabla u\|_{L^2}^2\leq C_3\left(\|D^2 u\|_{L^2}^\frac{1}{2}\|u\|_{L^2}^\frac{1}{2}+\|u\|_{L^2}\right)^2\leq &C_4\left(\|D^2 u\|_{L^2}^\frac{1}{2}+1\right)^2\\
&\leq \|\Delta u\|_{L^2}^2+C_5.
\end{split}
\ee
Substituting \eqref{grad-gn} into \eqref{grad U-diff},  we end up with
$$
\frac{d}{dt}\int_\Omega |\nabla u|^2+\int_\Omega |\nabla u|^2\leq C_6,
$$
yielding directly the unform boundedness of $\|\nabla u\|_{L^2}$. The same type argument applied to the $w$-equation in  \eqref{Loop equations} gives the unform boundedness of $\|\nabla w\|_{L^2}$.

Now, we again use the $u$-equation in  \eqref{Loop equations} to calculate that
\be\label{grad u-L4-diff}
\begin{split}
&\frac{1}{2}\frac{d}{dt}\int_\Omega |\nabla u|^4+\int_\Omega \left|\nabla |\nabla u|^2\right|^2+2\int_\Omega |\nabla u|^2|D^2u|^2\\
&\ =-2\chi_1\int_\Omega |\nabla u|^2\nabla u\cdot\nabla\left(\nabla u\nabla v+u\Delta v\right)+\int_{\partial \Omega} |\nabla  u|^2\frac{\partial}{\partial \nu}|\nabla u|^2.
\end{split}
\ee
By direct computations, we discover that
\be\label{gradgrad-com}
 \nabla(\nabla u\nabla v)=u_{x_1}\nabla v_{x_1}+v_{x_1}\nabla u_{x_1}+u_{x_2}\nabla v_{x_2}+v_{x_2}\nabla u_{x_2}.
\ee
With these, we then employ \eqref{grad u-L4-diff} and \eqref{u+w-Linfty-bdd} to estimate, for any $\epsilon_i>0$, that
\begin{equation*}
\begin{split}
&-2\chi_1\int_\Omega |\nabla u|^2 u_{x_1}\nabla u\nabla v_{x_1}\leq \epsilon_1\int_\Omega |\nabla u|^6+C_{\epsilon_1} \int_\Omega |\nabla v_{x_1}|^3, \\
&-2\chi_1\int_\Omega |\nabla u|^2 v_{x_1}\nabla u\nabla u_{x_1}\leq \epsilon_2 \int_\Omega |\nabla u|^2|D^2 u|^2+C_{\epsilon_2} \int_\Omega |\nabla u|^4, \\
&-2\chi_1\int_\Omega |\nabla u|^2 u_{x_2}\nabla u\nabla v_{x_2}\leq \epsilon_3\int_\Omega |\nabla u|^6+C_{\epsilon_3} \int_\Omega |\nabla v_{x_2}|^3, \\
&-2\chi_1\int_\Omega |\nabla u|^2 v_{x_2}\nabla u\nabla u_{x_2}\leq \epsilon_4 \int_\Omega |\nabla u|^2|D^2 u|^2+C_{\epsilon_4} \int_\Omega |\nabla u|^4
\end{split}
\end{equation*}
and further that
\begin{equation*}
\begin{split}
&-2\chi_1\int_\Omega |\nabla u|^4\Delta v\leq \epsilon_5\int_\Omega |\nabla u|^6+C_{\epsilon_5} \int_\Omega |\Delta v|^3,\\
&-2\chi_1\int_\Omega u|\nabla u|^2\nabla u \nabla \Delta v\leq \epsilon_6\int_\Omega |\nabla u|^6+C_{\epsilon_6}\int_\Omega |\nabla \Delta v|^2\\
&C_{\epsilon_1} \int_\Omega |\nabla v_{x_1}|^3+C_{\epsilon_3} \int_\Omega |\nabla v_{x_2}|^3\\
&\leq C_{\epsilon_1,\epsilon_3}  \int_\Omega |D^2 v|^3+C_{\epsilon_1,\epsilon_3} \leq \tilde{C}_{\epsilon_1,\epsilon_3}  \int_\Omega |\Delta v|^3+\tilde{C}_{\epsilon_1,\epsilon_3} \leq \int_\Omega |\nabla \Delta v|^2+\hat{C}_{\epsilon_1,\epsilon_3},
\end{split}
\end{equation*}
where we have applied the uniform boundedness of $\| v\|_{H^2}$,   the 2D G-N interpolation inequality  and the $W^{2,3}$-elliptic estimate  in the last estimate.  Here and below, by $W^{i+2,p}$-elliptic estimate with $i\in \mathbb{N}, p>1$ (cf. \cite{ADN59,ADN64, Ladyzenskaja710}), we mean there exists $C_{i,p}>0$ such that, for any  function $f\in W^{i+2,p}(\Omega)$,  it follows
$$
\|f\|_{W^{i+2,p}}\leq C_{i,p}\left(\|\Delta f\|_{W^{i,p}}+\|f\|_{L^p}\right).
$$
In view of the uniform boundedness of $\|\nabla u\|_{L^2}$ and the fact fact $\frac{\partial u}{\partial \nu}=0$ on $\partial \Omega$,  by the 2D G-N inequality and boundary trace  embedding, the following two estimates are quite known (cf. \cite[(3.31) and (3.32)]{Xiang18-JMP} for example):
\begin{equation*}
 \int_\Omega |\nabla u|^6\leq C_7\int_\Omega |\nabla|\nabla u|^2|^2+C_7, \  \ \ \ \
\int_{\partial \Omega} |\nabla  u|^2\frac{\partial}{\partial \nu}|\nabla u|^2\leq \epsilon_7 \int_\Omega |\nabla|\nabla u|^2|^2+C_{\epsilon_7}.
\end{equation*}
Substituting these estimates into \eqref{grad u-L4-diff} and choosing sufficiently small $\epsilon_i$, we infer
\be\label{grad u-L4-diff-fin}
\frac{d}{dt}\int_\Omega |\nabla u|^4+\int_\Omega |\nabla u|^4\leq C_8\int_\Omega |\nabla \Delta v|^2+C_8.
\ee
 Finally, we combine \eqref{grad u-L4-diff-fin} with \eqref{u+w-l2-diff} to derive an ODI as follows:
$$
\frac{d}{dt}\int_\Omega \left(|\nabla u|^4+C_8\tau_1|\Delta v|^2\right)+\int_\Omega \left(|\nabla u|^4+2C_8|\Delta v|^2\right)\leq C_9,
$$
which trivially yields the uniform  boundedness of $\|\nabla u\|_{L^4}$. Doing the same argument to the $w$-equation in  \eqref{Loop equations} shows the unform boundedness of $\|\nabla w\|_{L^4}$.

Now, we once again use the $u$-equation in \eqref{Loop equations} and use \eqref{delta vl3-bdd by},  \eqref{u+w-Linfty-bdd},  \eqref{gradgrad-com} and the elliptic $H^2$-estimate  to bound
\be\label{delta u-diff}
\begin{split}
&\frac{d}{dt}\int_\Omega |\Delta u|^2+\int_\Omega |\Delta u|^2+2\int_\Omega |\nabla \Delta u|^2\\
&=2\chi_1\int_\Omega \nabla\left(\nabla u\nabla v+u\Delta v\right)\nabla\Delta u+\int_\Omega |\Delta u|^2\\
&\leq \int_\Omega |\nabla \Delta u|^2+C_{10}\int_\Omega |D^2 u|^2+C_{10}\int_\Omega |\nabla  u|^4+\int_\Omega |\Delta u|^2\\
&\ \ +C_{10}\int_\Omega |D^2 v|^2 +C_{10}\int_\Omega |\Delta v|^4 +C_{10}\int_\Omega |\nabla \Delta v|^2\\
&\leq \int_\Omega |\nabla \Delta u|^2+C_{11}\int_\Omega |\Delta u|^2+C_{11}\int_\Omega |\Delta  v|^4+C_{11}\int_\Omega |\nabla \Delta v|^2+C_{11}\\
&\leq 2\int_\Omega |\nabla \Delta u|^2+C_{12}\int_\Omega |\nabla \Delta v|^2+C_{12},
\end{split}
\ee
where we have used the uniform $(H^1, H^2)$-boundness of $(u,v)$  in the last two lines. Combining \eqref{delta u-diff}    with \eqref{u+w-l2-diff}, we then derive a key ODI as follows:
$$
\frac{d}{dt}\int_\Omega \left(|\Delta u|^2+C_{12}\tau_1|\Delta v|^2\right)+\int_\Omega \left(|\Delta u|^2+2C_{12}|\Delta v|^2\right)\leq C_{13},
$$
showing the   uniform  boundedness of $\|\Delta u\|_{L^2}$. The same argument applied to the $w$-equation in  \eqref{Loop equations} shows the unform boundedness of $\|\Delta w\|_{L^2}$. In light of our gained uniform $H^1$-boundedness of $(u,w)$, the $W^{2,2}$-elliptic estimate and the 2D Sobolev embedding $W^{2,2}(\Omega)\hookrightarrow W^{1, q}(\Omega)$ for all $q\in(1,\infty)$, we end up with
\be\label{uw-h2-est}
\begin{cases}
\|u(t)\|_{W^{2,2}}+\|w(t)\|_{W^{2,2}}\leq C_{14},  \  t\geq 1, \\[0.2cm]
  q<\infty, \   \|u(t)\|_{W^{1, q}}+\|w(t)\|_{W^{1, q}}\leq C_{q},  \  t\geq1.
\end{cases}
\ee
 Now, taking $\Delta $ operator and normal derivative  on the $v$ and $z$ equations in  \eqref{Loop equations} and using the facts $\nu\cdot\nabla v|_{\partial \Omega}=0=\nu\cdot\nabla z|_{\partial \Omega}$, we discover that
\be\label{grad-eqs2}
\begin{cases}
\tau_1 (\Delta v)_t = \Delta (\Delta v)-\Delta v+ \Delta w  &\text{in } \Omega\times(1,\infty), \\[0.2cm]
\tau_2 (\Delta z)_t = \Delta (\Delta z)-\Delta z+ \Delta u &\text{in } \Omega\times(1,\infty), \\[0.2cm]
\frac{\partial \Delta v}{\partial \nu}=\frac{\partial \Delta z}{\partial \nu}=0&\text{on } \partial\Omega\times(1,\infty).
\end{cases}
\ee
By the facts that $\left(\Delta u(t), \  \Delta w(t)\right) \in \left(C(\bar{\Omega})\right)^2$, the standard Schauder regularity says that
$\left(\Delta v(t), \  \Delta z(t)\right) \in \left(C^2(\bar{\Omega})\right)^2$. Therefore,  applying $W^{2,2}$-estimate or $W^{1,q}$-estimate in Lemma  \ref{reciprocal-lem}  to  \eqref{grad-eqs2} , we obtain   that
\be\label{vz-dradv2-est}
\begin{cases}
\|\Delta v(t)\|_{W^{2,2}}+\|\Delta z(t)\|_{W^{2,2}}\leq C_{15},  \  t\geq 1, &\text{ if  } \tau_1=\tau_2=0, \\[0.2cm]
q>1,
  \ \|\Delta v(t)\|_{W^{1, q}}+\|\Delta z(t)\|_{W^{1, q}}\leq C_q,  \  t\geq1, &\text{ if  } \tau_1, \tau_2>0.
\end{cases}
\ee
Thus,  by  $W^{4,2}$-elliptic estimate if $\tau_1=\tau_2=0$ or  $W^{3,3}$-elliptic estimate  if $\tau_1, \tau_2>0$  and Sobolev embeddings $W^{4,2}(\Omega)\hookrightarrow W^{3,3}(\Omega)\hookrightarrow W^{2,\infty}(\Omega)$, there exists $C_{16}>0$ such that
\be\label{vz-c2-est}
\|v(t)\|_{W^{2, \infty}}+\|z(t)\|_{W^{2, \infty}}\leq C_{16}, \quad  t\geq 1.
\ee
Progressively, we again utilize the $u$-equation in \eqref{Loop equations} and utilize  \eqref{delta vl3-bdd by},  \eqref{u+w-Linfty-bdd}, \eqref{uw-h2-est}, \eqref{vz-dradv2-est} and \eqref{vz-c2-est}   to bound
\be\label{delta ul4-diff}
\begin{split}
&\frac{1}{2}\frac{d}{dt}\int_\Omega |\Delta u|^4+\int_\Omega |\Delta u|^4+6\int_\Omega |\Delta u|^2|\nabla \Delta u|^2\\
&=6\chi_1\int_\Omega (\Delta u)^2\nabla \left(\nabla u\nabla v+u\Delta v\right)\nabla(\Delta u)+\int_\Omega |\Delta u|^4\\
&\leq  \int_\Omega |\Delta u|^2|\nabla \Delta u|^2+\int_\Omega |\Delta u|^4+C_{17}\int_\Omega |\Delta u|^2|\nabla  u|^2\\
&\ \ \ +C_{17}\int_\Omega |\Delta u|^2|D^2  u|^2 +C_{17}\int_\Omega |\Delta u|^2|\nabla \Delta v|^2\\
&\leq  \int_\Omega |\Delta u|^2|\nabla \Delta u|^2+C_{18}\int_\Omega |\nabla u|^4+C_{18}\int_\Omega |\Delta u|^4 \\
&\quad   +C_{18}\int_\Omega |D^2 u|^4+C_{18}\int_\Omega |\nabla \Delta v|^4\\
&\leq  \ \  \int_\Omega |\Delta u|^2|\nabla \Delta u|^2+C_{19}\int_\Omega |\Delta u|^4+C_{19}\int_\Omega |D^2 u|^4+C_{19}.
\end{split}
\ee
We next use  the uniform boundedness of $\|u\|_{H^2}$, the $W^{2, 4}$-elliptic estimate and the 2D G-N interpolation inequality in \eqref{G-N-I} to infer, for any $\epsilon>0$,  that
\begin{equation*}
\begin{split}
\|D^2u\|_{L^4}^4&\leq C_{20}\left(\|\Delta u\|_{L^4}+\|u\|_{L^4}\right)^4\\
 &\leq C_{21}\left(\|\Delta u\|_{L^4}^4+1\right)\\
 &\leq  C_{22}\left(\|\nabla (\Delta u)^2\|_{L^2}
 \|(\Delta u)^2\|_{L^1}+\|(\Delta u)^2\|_{L^1}^2\right)\leq \epsilon\int_\Omega |\Delta u|^2|\nabla \Delta u|^2+C_\epsilon.
\end{split}
\end{equation*}
This along   with \eqref{delta ul4-diff} enables us to see that
$$
\frac{d}{dt}\int_\Omega |\Delta u|^4+\int_\Omega |\Delta u|^4\leq C_{24},
$$
showing the   uniform  boundedness of $\|\Delta u\|_{L^4}$. The same argument applied to the $w$-equation in  \eqref{Loop equations} entails the unform boundedness of $\|\Delta w\|_{L^4}$. Due to our established uniform $H^2$-boundedness of $(u,w)$, the $W^{2,4}$-elliptic estimate and the 2D Sobolev embedding $W^{2,4}(\Omega)\hookrightarrow W^{1, \infty}(\Omega)$,  we finally  conclude  that
\be\label{uw-w24-est}
\begin{cases}
\|u(t)\|_{W^{2,4}}+\|w(t)\|_{W^{2,4}}\leq C_{25},  \  t\geq 1, \\[0.2cm]
  \|u(t)\|_{W^{1, \infty}}+\|w(t)\|_{W^{1, \infty}}\leq C_{26},  \  t\geq1.
\end{cases}
\ee
Then applying $W^{2,4}$-estimate or $W^{1,q}$-estimate in Lemma  \ref{reciprocal-lem}  to  \eqref{grad-eqs2} , we see that
 \be\label{vz-c3-est}
\begin{cases}
\|\Delta v (t)\|_{W^{2,4}}+\|\Delta z(t)\|_{W^{2,4}}\leq C_{27},  \  t\geq 1, &\text{ if  } \tau_1=\tau_2=0, \\[0.2cm]
  \ \|\Delta v(t)\|_{W^{1, \infty}}+\|\Delta z(t)\|_{W^{1, \infty}}\leq C_{28},  \  t\geq1, &\text{ if  } \tau_1, \tau_2>0.
\end{cases}
\ee
 Consequently, by $W^{4,4}$-elliptic estimate  and  the Sobolev embedding  $W^{4,4}(\Omega)\hookrightarrow W^{3,\infty}(\Omega)$ if $\tau_1=\tau_2=0$ or $W^{3,\infty}$-elliptic estimate  if $\tau_1, \tau_2>0$,   our desired higher order estimate \eqref{gradUW-improve-thm} follows from  \eqref{uw-w24-est} and \eqref{vz-c3-est}.
\end{proof}

\subsection{Finite time blow-up} In this subsection, we show (B4) by detecing a line of masses on which the solution  of \eqref{Loop equations} blows up in finite time under \eqref{large mass blow-up}.
\begin{lemma} \label{link} Assume that  $\tau_1= \tau_2$ and $\chi_3=0$ and  assume that $(n,c)$ solves
\be \label{nc-1}\begin{cases}
n_t =  \nabla\cdot\left(\nabla n-\chi_1n\nabla c\right) &\text{in } \Omega\times(0,\infty), \\[0.2cm]
\tau_1 c_t = \Delta c-c+\frac{\chi_2}{\chi_1}n  &\text{in } \Omega\times(0,\infty), \\[0.2cm]
\frac{\partial n}{\partial \nu}=\frac{\partial c}{\partial \nu}=0&\text{on } \partial\Omega\times(0,\infty), \\[0.2cm]
\left(n, \   \tau_1 c\right)=\left(u_0,\   \tau_1 v_0\right)&\text{in } \Omega\times\{0\}.
\end{cases}
\ee
Then, if $(w_0,\tau_2z_0)=(\frac{\chi_2}{\chi_1}u_0, \frac{\chi_1}{\chi_2}\tau_1v_0)$,  the unique solution of  \eqref{Loop equations} on its maximal existence time is given by $(u,v,w,z)=(n,c,\frac{\chi_2}{\chi_1}n,\frac{\chi_1}{\chi_2}c )$.
\end{lemma}
\begin{proof}
By direct computations, one sees first that $(u,v,w,z)=(n,c,\frac{\chi_2}{\chi_1}n,\frac{\chi_1}{\chi_2}c )$ solves  \eqref{Loop equations}, and then, it is the unique solution of  \eqref{Loop equations} by uniqueness.
\end{proof}
Based on this observation, we use the well-known blow-up knowledge about \eqref{nc-1} to detect a blow-up line for our two-species and two-stimuli model \eqref{Loop equations}.

\begin{lemma} \label{blowup-lem} Let $\tau_1= \tau_2, \chi_3=0$ and  $( w_0,\ \tau_2z_0)=(\frac{\chi_2}{\chi_1}u_0, \ \frac{\chi_1}{\chi_2}\tau_1v_0)$. Assume that  \eqref{large mass blow-up} is satisfied  and  $\int_\Omega u_0(x)|x-x_0|^2dx $ is sufficiently small for $x_0\in \bar\Omega$. Then the solution of the IBVP \eqref{Loop equations} blows up in a finite time $T>0$ according to \eqref{blowup-new}.
\end{lemma}
\begin{proof}
It follows from  $m_2=\int_\Omega w_0=\int_\Omega \frac{\chi_2}{\chi_1}u_0=\frac{\chi_2}{\chi_1}m_1$ that  $
m_1m_2\chi_1\chi_2=\left(m_1\chi_2\right)^2$. Thus, the large product condition \eqref{large mass blow-up} directly gives  $m_1\chi_2>\pi^*$. By the well-known  blow-up results about \eqref{nc-1} (cf. \cite{HV97, HW01, Horst-03, JL92, A11,A12,A13}), we know that the solution $(n,c)$ of \eqref{nc-1} blows up in a finite time $T>0$ in the sense that
$$
\limsup_{t\nearrow T}\left(\|n(t)\|_{L^\infty}+\|c(t)\|_{L^\infty}\right)=\infty.
$$
Then  Lemma \ref{link} simply  says that $(u,v,w,z)=(n,c,\frac{\chi_2}{\chi_1}n,\frac{\chi_1}{\chi_2}c )$ is a classical solution of \eqref{Loop equations} on $\bar{\Omega}\times[0, T)$ which blows up at $t=T$ even in the sense of \eqref{blowup-new}. Otherwise, the uniform $L^1$-boundedness of $(u\ln u, w\ln w)$ implies global boundedness (and thus no blow-up) by previous subsections, cf. Lemma \ref{gradUW-improve-lem}.
\end{proof}

\section{Exponential convergence for small  product of masses}

So far, we have proved the global boundedness of solutions to the IBVP  \eqref{Loop equations} under certain  smallness assumption on the product of masses and blow-up for certain large product.  In this section, we turn our attention to study the large time behavior of bounded solutions under (B3). To this end, we  find that it is more convenient to work on its equivalent system:
\begin{eqnarray}\label{equal system}
\begin{cases}
U_t =  \nabla\cdot\left(\nabla U-U\nabla V\right) &\text{in } \Omega\times(0,\infty), \\[0.2cm]
\tau_1 V_t = \Delta V-V+\eta_1(W-1)  &\text{in } \Omega\times(0,\infty), \\[0.2cm]
W_t =  \nabla\cdot\left(\nabla W-W\nabla Z-\chi W\nabla V\right) &\text{in } \Omega\times(0,\infty), \\[0.2cm]
\tau_2 Z_t = \Delta Z-Z+\eta_2(U-1)&\text{in } \Omega\times(0,\infty), \\[0.2cm]
\frac{\partial U}{\partial \nu}=\frac{\partial V}{\partial \nu}=\frac{\partial W}{\partial \nu}=\frac{\partial Z}{\partial \nu}=0&\text{on } \partial\Omega\times(0,\infty), \\[0.2cm]
\left(U, \   \tau_1V,\  W, \    \tau_2 Z\right)=\left(U_0,\   \tau_1 V_0,\    W_0, \  \tau_2 Z_0\right)&\text{in } \Omega\times\{0\}.
\end{cases}
\end{eqnarray}
Here, the newly introduced variables  satisfy the following transformations:
\begin{equation}\label{trans-id}
\begin{cases}
&U=\frac{u}{\bar{u}_0},\ \ \ U_0=\frac{u_0}{\bar{u}_0},\ \ \  V=\chi_1(v-\bar v), \ \ \  V_0=\chi_1(v_0-\bar{v}_0),\\[0.2cm]
&W=\frac{w}{\bar{w}_0}, \ \ \  W_0=\frac{w_0}{\bar{w}_0}, \ \ \   Z=\chi_2(z-\bar z), \ \ \ Z_0=\chi_2(z_0-\bar{z}_0),\\[0.2cm]
&\chi=\frac{\chi_3}{\chi_1}, \ \ \  \  \eta_1=\chi_1\bar{w}_0,\ \ \ \ \eta_2=\chi_2\bar{u}_0.
\end{cases}
\end{equation}
Then it follows simply from \eqref{equal system} and \eqref{trans-id} that
\be\label{UVWZ-int}
\bar U=1=\bar W,\quad \quad \quad  \bar V=0=\bar Z.
\ee
Then,  with $k$ defined in \eqref{k-def}, an easy use of Lemma \ref{k-const} shows  that
\be\label{UW-l2}
\begin{cases}
&\|U-1\|_{L^2}^2\leq k\|\nabla U^\frac{1}{2}\|_{L^2}^2,\\[0.25cm]
&\|W-1\|_{L^2}^2\leq k\|\nabla W^\frac{1}{2}\|_{L^2}^2.
\end{cases}
\ee
Let us begin with the simpler case when $\tau_1=\tau_2=0$. In this case, our convergence will rely on building a conditional  Lyapunov functional of the form:
 \be\label{F-def-con-pe}
\mathcal{G}(t)=\frac{\eta_2}{\eta_1}\int_{\Omega}U\ln U+\int_{\Omega}W\ln W.
 \ee
To that purpose,   we multiply  the first equation  by $\ln U$ and the third equation by  $\ln W$ in \eqref{equal system} and then integrate over $\Omega$  by parts to infer
  \be\label{UlnU-id-pe}
\begin{split}
&\frac{d}{dt}\int_{\Omega}U\ln U+4\int_{\Omega}|\nabla U^{\frac{1}{2}}|^2\\
&=\eta_1\int_{\Omega}(U-1)(W-1)-\int_{\Omega}(U-1)V
\end{split}
\end{equation}
and, similarly,
\begin{equation} \label{WlnW-id-pe}
\begin{split}
\frac{d}{dt}\int_\Omega W\ln W+4\int_{\Omega}|\nabla W^{\frac{1}{2}}|^2&=\eta_2\int_{\Omega}(U-1)(W-1)-\int_{\Omega}(W-1)Z\\
&\  \ +\eta_1\chi\int_{\Omega}(W-1)^2-\chi\int_{\Omega}(W-1)V.
\end{split}
\end{equation}
Thanks to $\tau_1=\tau_2=0$,  we see from the second and fourth equation in \eqref{equal system} that
\be\label{VZ-l2-by}
\begin{cases}
2\int_\Omega |\nabla V|^2+2\int_\Omega V^2=2\eta_1\int_\Omega (W-1)V\leq\int_\Omega V^2+\eta_1^2\int_\Omega (W-1)^2,\\[0.25cm]
2\int_\Omega |\nabla Z|^2+2\int_\Omega Z^2=2\eta_2\int_\Omega (U-1)Z\leq\int_\Omega Z^2+\eta_2^2\int_\Omega (U-1)^2 .
\end{cases}
\ee
With those computations, we next derive the derivative of $\mathcal{G}$  and its decay property.
 \begin{lemma}  \label{con-lem-pe} When $\tau_1=\tau_2=0$, the derivative of  $\mathcal{G}$  defined in \eqref{F-def-con-pe} satisfies
\begin{equation}\label{F-diif-con-pe}
\begin{split}
&\mathcal{G}'(t)+\frac{4\eta_2}{\eta_1}\int_\Omega|\nabla U^{\frac{1}{2}}|^2+4\int_{\Omega}|\nabla W^{\frac{1}{2}}|^2\\
&=2\eta_2\int_\Omega(U-1)(W-1)-\frac{\eta_2}{\eta_1}\int_{\Omega}(U-1)V\\
&\ -\int_{\Omega}(W-1)Z
+\eta_1\chi\int_{\Omega}(W-1)^2-\chi\int_{\Omega}(W-1)V, \ \ t\in [0, T_m).
\end{split}
\end{equation}
Moreover, if
\be\label{eta1-eta2-small-pe}
k^2\eta_1\eta_2+k\eta_1\chi^+<4,\quad \chi^+=\max\{\chi, \ \ 0\},
 \ee
  then $\mathcal{G}$ decays exponentially according to
\be\label{f-decay-pe}
0\leq \mathcal{G}(t)\leq \mathcal{G}(0)e^{-\mu t},\quad t\in [ 0, T_m),
\ee
where the positive and explicit decay rate $\mu$ is given by
\be\label{decay-rate1}
\mu=\frac{\left(4-k^2\eta_1\eta_2-k\eta_1\chi^+\right)}{2k^2}\min\left\{k, \ \ \ \frac{4}{k^2\eta_1\eta_2+2\left(4-k^2\eta_1\eta_2-k\eta_1\chi^+\right)}\right\}.
\ee
A direct consequence of \eqref{f-decay-pe} and \eqref{F-def-con-pe} is
\be\label{uw-ln-exp}
\eta_2\int_\Omega U\ln U+ \eta_1\int_\Omega W\ln W \leq \eta_1\mathcal{G}(0) e^{-\mu t}, \ \ \  t\in [ 0, T_m).
\ee
 \end{lemma}
\begin{proof}A simple linear combination from \eqref{UlnU-id-pe}, \eqref{WlnW-id-pe} and \eqref{VZ-l2-by} entails \eqref{F-diif-con-pe}.

Next, since $s\ln s\geq s-1$ for any $s>0$, upon integration,  we obtain from the facts   $\bar U=\bar W=1$ that
\be\label{f nonnegative}
\int_{\Omega}U\ln U \geq \int_{\Omega} \left(U-1\right)=0, \quad \quad \int_{\Omega}W\ln W \geq \int_{\Omega} \left(W-1\right)=0.
\ee
which together with the definition of $\mathcal{G}$ in \eqref{F-def-con-pe} immediately shows $\mathcal{G}\geq 0$.

Now,  employing repeatedly Young's inequality  with epsilon  and \eqref{VZ-l2-by}, for $\epsilon_i>0$ to be fixed as \eqref{epsiloni-choices}, we bound the right-hand of \eqref{F-diif-con-pe} term by term as
\begin{equation*}
\begin{cases}
&2\eta_2\int_\Omega(U-1)(W-1)\leq \epsilon_1\int_\Omega(U-1)^2+\frac{\eta_2^2}{\epsilon_1}\int_\Omega(W-1)^2,\\[0.2cm]
&-\frac{\eta_2}{\eta_1}\int_{\Omega}(U-1)V\leq \epsilon_2\int_{\Omega}(U-1)^2+\frac{\eta_2^2}{4\eta_1^2\epsilon_2}\int_{\Omega} V^2,\\[0.2cm]
&\quad \quad \quad \quad \quad \quad \quad \leq \epsilon_2\int_{\Omega}(U-1)^2+\frac{\eta_2^2}{4\epsilon_2}\int_{\Omega} (W-1)^2
\end{cases}
\end{equation*}
and further
\begin{equation*}
\begin{cases}
&-\int_{\Omega}(W-1)Z\leq \frac{\epsilon_3}{\eta_2^2}\int_{\Omega} Z^2+\frac{\eta_2^2}{4\epsilon_3}\int_{\Omega}(W-1)^2,\\[0.2cm]
&\quad \quad \quad \quad \quad \quad \quad\leq \epsilon_3\int_{\Omega} (U-1)^2+\frac{\eta_2^2}{4\epsilon_3}\int_{\Omega}(W-1)^2,\\[0.2cm]
&-\chi\int_{\Omega}(W-1)V\leq\frac{\eta_1\chi^-}{2}\int_{\Omega}(W-1)^2+\frac{\chi^-}{2\eta_1}\int_{\Omega} V^2
-\frac{\chi^+}{\eta_1}
\int_{\Omega} V^2 \\[0.2cm]
&\quad \quad \quad \quad \quad\quad \leq\frac{\eta_1\chi^-}{2}\int_{\Omega}(W-1)^2+\left(\frac{\chi^-}{2\eta_1}-\frac{\chi^+}{\eta_1}\right)^+
\eta_1^2\int_{\Omega} (W-1)^2\\
&\quad \quad \quad \quad \quad\quad =\eta_1\chi^-\int_{\Omega} (W-1)^2.
\end{cases}
\end{equation*}
Combining these inequalities  with  \eqref{UW-l2}, we bound \eqref{F-diif-con-pe} as follows:
\begin{equation}\label{F-diif-con-pe0}
\begin{split}
&\mathcal{G}'(t)+\frac{4\eta_2}{\eta_1}\int_\Omega|\nabla U^{\frac{1}{2}}|^2+4\int_{\Omega}|\nabla W^{\frac{1}{2}}|^2\\
&\leq \left(\epsilon_1+\epsilon_2+\epsilon_3\right)\int_\Omega (U-1)^2 +\left(\frac{\eta_2^2}{\epsilon_1}+\frac{\eta_2^2}{4\epsilon_2}
+\frac{\eta_2^2}{4\epsilon_3}+\eta_1\chi^+\right)\int_\Omega (W-1)^2\\
&\leq \left(\epsilon_1+\epsilon_2+\epsilon_3\right)k\int_\Omega|\nabla U^{\frac{1}{2}}|^2+\left(\frac{\eta_2^2}{\epsilon_1}+\frac{\eta_2^2}{4\epsilon_2}
+\frac{\eta_2^2}{4\epsilon_3}+\eta_1\chi^+\right)k\int_\Omega|\nabla W^{\frac{1}{2}}|^2.
\end{split}
\end{equation}
Thanks to \eqref{eta1-eta2-small-pe}, we now  fix, for instance,
\be\label{epsiloni-choices}
 \epsilon_1=\frac{2\eta_2}{k\eta_1}, \quad\quad
 \epsilon_2=\frac{\eta_2}{k\eta_1},  \quad\quad
 \epsilon_3=\left[\frac{k\eta_1}{\eta_2}
 +\frac{2(4-k^2\eta_1\eta_2-k\eta_1\chi^+)}{k\eta_2^2}\right]^{-1}
\ee
so that
\be\label{AB-choices}
\begin{cases}
A&:=\frac{4\eta_2}{\eta_1}-\left(\epsilon_1+\epsilon_2+\epsilon_3\right)k\\[0.2cm]
&=\frac{2\eta_2\left(4-k^2\eta_1\eta_2-k\eta_1\chi^+\right)}{k\eta_1
\left[k^2\eta_1\eta_2+2\left(4-k^2\eta_1\eta_2-k\eta_1\chi^+\right)\right]}>0, \\[0.25cm]
\hat{A}&:=4-\left(\frac{\eta_2^2}{\epsilon_1}+\frac{\eta_2^2}{4\epsilon_2}
+\frac{\eta_2^2}{4\epsilon_3}+\eta_1\chi^+\right)k\\[0.2cm]
&=\frac{1}{2}\left(4-k^2\eta_1\eta_2-k\eta_1\chi^+\right)>0.
\end{cases}
\ee
Next, by the algebraic  inequality
$$
s\ln s \leq s-1+(s-1)^2, \quad \quad \forall s>0,
$$
upon directly integrating and using \eqref{UW-l2} and the facts that $\bar U=\bar W=1$, we find
\begin{equation}\label{UlnU+WlnW-by UWl2}
\begin{cases}
&\int_\Omega U\ln U\leq\int_\Omega(U-1)+\int_\Omega\left(U-1\right)^2=\int_\Omega \left(U-1\right)^2\leq k\int_\Omega |\nabla U^\frac{1}{2}|^2, \\[0.25cm]
&\int_\Omega W\ln W \leq \int_\Omega \left(W-1\right)^2\leq k\int_\Omega |\nabla W^\frac{1}{2}|^2.
\end{cases}
\end{equation}
Substituting \eqref{UlnU+WlnW-by UWl2} into \eqref{F-diif-con-pe0} and recalling \eqref{F-def-con-pe} and \eqref{AB-choices}, we finally end up with  a key ODI for $\mathcal{G}$ as follows:
 $$
\mathcal{G}'(t)\leq -\min\left\{\frac{A\eta_1}{k\eta_2}, \ \ \frac{\hat{A}}{k}\right\}\mathcal{G}(t), \quad t\in [ 0, T_m),
$$
which along with \eqref{AB-choices} gives rise to \eqref{f-decay-pe} with $\mu$ given by \eqref{decay-rate1}.
\end{proof}
\begin{remark} One can easily check the proof of   Lemma \ref{con-lem-pe} works simply for the  limiting case of  $\eta_2=0$. Thus, by setting $U=Z\equiv0$ formally,   the $L^1$-convergence $W\ln W$ for the minimal  KS model holds under $k\eta_1\chi^+<4$.

If
 $$
k^2\eta_1\eta_2+k\eta_1\chi^+\leq 4,
$$
  it follows easily from the proof of this lemma that  $\|U\ln U\|_{L^1}+\|W\ln W\|_{L^1}$ is uniformly bounded on $(0, T_m)$, and then our Section 3 implies that the solution to \eqref{equal system} or equivalently \eqref{Loop equations}  exists globally in time and is bounded on $\Omega\times (0, \infty)$.
\end{remark}
Next, we explore the convergence property for the case that $\tau_1>0$ and $\tau_2>0$.
For this purpose,  we shall construct a conditional Lyapunov  functional of the form (cf. \cite{GZ98,A19,A20} in other situations):
\be\label{energy function}
\begin{split}
\mathcal{H}(t)=&\frac{\alpha}{k}\int_{\Omega}U\ln U+\frac{\tau_1\alpha}{2}\int_{\Omega}|\nabla V|^2+\frac{\tau_1\alpha}{2}\left
(1+2 \beta +\frac{\gamma_1}{k\eta_1}\right)\int_{\Omega}V^2\\
&+\frac{1}{k}\int_{\Omega}W\ln W+\frac{\tau_2}{2}\int_{\Omega}|\nabla Z|^2+\frac{\tau_2}{2}\left
(1+2 \beta +\frac{\gamma_2}{k\eta_2}\right)\int_{\Omega}Z^2,
\end{split}
\ee
where  $k$ is given in \eqref{UW-l2} and nonnegative $\alpha,\beta, \gamma_i \ (i=1,2)$  will be detailed in \eqref{para-choice}.

\begin{lemma}  The time derivative of   $\mathcal{H}$ defined in \eqref{energy function}  fulfills
\begin{equation}\label{diff-G}
\begin{split}
&\mathcal{H}'(t)+\frac{4\alpha}{k}\int_{\Omega}|\nabla U^{\frac{1}{2}}|^2+\frac{4}{k}\int_{\Omega}|\nabla W^{\frac{1}{2}}|^2+\frac{\gamma_1\alpha}{k\eta_1}\int_{\Omega}\left(V^2+|\nabla V|^2\right)\\
&\ \ \ +\tau^2_1\alpha\int_{\Omega}V_t^2+\frac{\gamma_2}{k\eta_2}\int_{\Omega}\left(Z^2+|\nabla Z|^2\right)+\tau^2_2 \int_{\Omega}Z_t^2\\
& =\frac{(\eta_1\alpha+\eta_2)}{k}\int_{\Omega}(U-1)(W-1)
-\frac{\tau_1\alpha}{k}\int_{\Omega}(U-1)V_t\\
& \ \  -\frac{\alpha}{k}\int_{\Omega}(U-1)V-\frac{\tau_2}{k}\int_{\Omega}(W-1)Z_t
-\frac{1}{k}\int_{\Omega}(W-1)Z\\
& \ \ +\frac{\eta_1\chi}{k}\int_{\Omega}(W-1)^2 +\tau_2\eta_2\int_{\Omega}(U-1)Z_t
+\frac{\gamma_2}{k}\int_{\Omega}(U-1)Z\\
&\ \ \ +\left(\eta_1\alpha-\frac{\chi}{k}\right)\tau_1\int_{\Omega}(W-1)V_t
+\frac{(\gamma_1\alpha-\chi)}{k}\int_{\Omega}(W-1)V\\
&\ \ \ +2 \beta \alpha\tau_1\int_\Omega VV_t+2 \beta \tau_2\int_\Omega ZZ_t, \ \ t\in [0, T_m).
\end{split}
\end{equation}
 \end{lemma}
\begin{proof}
By \eqref{equal system}  and \eqref{energy function}, we see that the differentiation of  $\mathcal{H}$ solves
\begin{equation}\label{writting form}
\begin{split}
\mathcal{H}'(t)=&\frac{\alpha}{k}\int_{\Omega}U_t\ln U+\tau_1\alpha\int_{\Omega}\nabla V_t\cdot\nabla V+\tau_1\alpha\left(1+2 \beta +\frac{\gamma_1}{k\eta_1}\right)\int_{\Omega}V_t V\\
&+\frac{1}{k}\int_{\Omega}W_t\ln W+\tau_2\int_{\Omega}\nabla Z_t\cdot\nabla Z+\tau_2\left(1+2 \beta +\frac{\gamma_2}{k\eta_2}\right)\int_{\Omega}Z_t Z.
\end{split}
\end{equation}
For the terms $\frac{\alpha}{k}\int_{\Omega}U_t\ln U$ and $\frac{1}{k}\int_{\Omega}W_t\ln W$, testing the first equation of \eqref{equal system} by $\frac{1}{k}\ln U$ and the third equation by  $\frac{1}{k}\ln W$, we conclude that
\begin{equation}\label{UlnU}
\begin{split}
&\frac{\alpha}{k}\int_{\Omega}U_t\ln U+\frac{4\alpha}{k}\int_{\Omega}|\nabla U^{\frac{1}{2}}|^2\\
&=-\frac{\alpha}{k}\int_{\Omega} (U-1)\cdot\Delta V\\
&=\frac{\eta_1\alpha}{k}\int_{\Omega}(U-1)(W-1)-\frac{\tau_1\alpha}{k}\int_{\Omega}(U-1)V_t
-\frac{\alpha}{k}\int_{\Omega}(U-1)V
\end{split}
\end{equation}
and, similarly,
\begin{equation} \label{WlnW-id}
\begin{split}
&\frac{}{k}\frac{d}{dt}\int_\Omega W\ln W+\frac{4}{k}\int_{\Omega}|\nabla W^{\frac{1}{2}}|^2\\
&=\frac{\eta_2}{k}\int_{\Omega}(U-1)(W-1)-\frac{\tau_2}{k}\int_\Omega(W-1)Z_t-\frac{1}{k}\int_{\Omega}(W-1)Z\\
&\ \ +\frac{\eta_1\chi}{k}\int_{\Omega}(W-1)^2-\frac{\tau_1\chi}{k}\int_\Omega(W-1)V_t-\frac{ \chi}{k}\int_{\Omega}(W-1)V.
\end{split}
\end{equation}
As to the second and third terms in \eqref{writting form},  we  use the second equation in \eqref{equal system}  to calculate  that
\begin{equation}\label{gradV-odi-pp}
\tau_1\alpha\int_{\Omega}\nabla V_t\cdot\nabla V=\tau_1\eta_1\alpha\int_{\Omega}\left(W-1\right)V_t-\tau^2_1\alpha\int_{\Omega} V_t^2- \tau_1\alpha\int_{\Omega}V_t V
\end{equation}
and that
\be\label{V-odi-pp}
\begin{split}
&\tau_1\alpha\left(1+2 \beta +\frac{\gamma_1}{k\eta_1}\right)\int_{\Omega}V_t V\\
&=(1+2 \beta )\tau_1\alpha\int_{\Omega}V_t V+\frac{\tau_1\gamma_1\alpha}{k\eta_1}\int_{\Omega}V_t V\\
&=(1+2 \beta )\tau_1\alpha\int_{\Omega}V_t V+\frac{\gamma_1\alpha}{k\eta_1}\int_{\Omega}\left[\Delta V- V+\eta_1\left(W-1\right)\right] V\\
&=(1+2 \beta )\tau_1\alpha\int_{\Omega}V_t V-\frac{\gamma_1\alpha}{k\eta_1}\int_{\Omega}\left(V^2+|\nabla V|^2\right)+\frac{\gamma_1\alpha}{k}\int_{\Omega}\left(W-1\right)V.
\end{split}
\ee
 In the same reasoning, we find that
\begin{equation}\label{GradZ-odi-pp}
\begin{split}
\tau_2\int_{\Omega}\nabla Z_t\cdot\nabla Z
&=\tau_2\eta_2\int_{\Omega}\left(U-1\right)Z_t-\tau^2_2\int_{\Omega} Z_t^2- \tau_2\int_{\Omega}Z_t Z
\end{split}
\end{equation}
and that
\begin{equation}\label{Z-odi-pp}
\begin{split}
&\tau_2\left(1+2 \beta +\frac{\gamma_2}{k\eta_2}\right)\int_{\Omega}Z_t Z\\
&=(1+2 \beta )\tau_2\int_{\Omega}Z_t Z-\frac{\gamma_2}{k\eta_2}\int_{\Omega}\left(Z^2+|\nabla Z|^2\right)+\frac{\gamma_2}{k}\int_{\Omega}\left(U-1\right)Z.
\end{split}
\end{equation}
Substituting \eqref{UlnU}, \eqref{WlnW-id}, \eqref{gradV-odi-pp}, \eqref{V-odi-pp},\eqref{GradZ-odi-pp}, \eqref{Z-odi-pp} into \eqref{writting form}, upon rearrangement, we finally accomplish  our stated dissipation identity \eqref{diff-G}.\end{proof}

Based on  \eqref{diff-G}, the function    $\mathcal{H}$ will decay  exponentially   under a smallness assumption on the product of initial masses, as provided below.
\begin{lemma} \label{con-lem-pp}
Let $\tau_1, \tau_2>0$,  $\chi\in\mathbb{R}$ and $\eta_1,\eta_2$ being from \eqref{trans-id} satisfy
 \begin{equation}\label{decay-rate-con-pp}
\begin{split}
\frac{2-\sqrt{22}}{3}<k\eta_1\chi<\sqrt{2}, \ \ \ k^2\eta_1\eta_2<\frac{2\sqrt{2}}{3}\min\left\{1, \ \ \frac{3}{2}+k\eta_1\chi\right\}.
\end{split}
\end{equation}
Then, for the specifications
\be\label{para-choice}
\begin{cases}
\alpha=\frac{1}{2}\left(\max\left\{\frac{\chi^2}{2}, \ \ \frac{\eta_2}{\sqrt{2}\eta_1}\right\}+\frac{3-2k\eta_1\chi^-}{3k^2\eta_1^2}\right), \\[0.25cm]
\beta =\frac{1-k^2\eta_1^2 \alpha }{k^2\eta_1^2 \alpha }, \ \ \gamma_1=\frac{k\eta_1\chi+1}{k\eta_1 \alpha }, \  \ \gamma_2=\frac{ \alpha }{k\eta_2},
\end{cases}
\ee
the function $\mathcal{H}$ defined in \eqref{energy function} decays exponentially according to
\be\label{decay property}
0\leq \mathcal{H}(t)\leq \mathcal{H}(0)e^{-\delta t}, \quad t\in [ 0, T_m).
\ee
Here, the rate $\delta=\delta(\eta_1,\eta_2, \tau_1,\tau_2,\chi, k)$ can be made precise as in \eqref{decay-rate} below.

A direct consequence from \eqref{decay property} and \eqref{energy function} follows: for some $C>0$,
\be\label{ulnu-wlnw-gradv-z-con}
\|U\ln U\|_{L^1}+\|W\ln W\|_{L^1}+\|V\|_{H^1}^2+\|Z\|_{H^1}^2\leq Ce^{-\delta t}, \quad \forall t\geq 0.
\ee
 \end{lemma}
\begin{proof}
Thanks to \eqref{decay-rate-con-pp},  we first find that
\be\label{odi-requirement3}
\begin{cases}
 \max\left\{\frac{\chi^2}{2}, \ \ \frac{\eta_2}{\sqrt{2}\eta_1}\right\}<\alpha<\frac{3+2k\eta_1\chi}{3k^2\eta_1^2},\\[0.2cm]
 \eta_1^2\alpha^2-3k^2 \eta_1^2 \eta_2^2\alpha+2\eta_2^2>0,\\[0.2cm]
 0< \alpha <\frac{1}{k^2\eta_1^2},
\end{cases}
\ee
which entails that $ \alpha, \beta, \gamma_1, \gamma_2$ defined in \eqref{para-choice} are positive, and so the function $\mathcal{H}$ is nonnegative by \eqref{f nonnegative} and \eqref{energy function}; moreover, along with  \eqref{para-choice}, \eqref{odi-requirement3} implies that
$$
\begin{cases}
&(1+ \beta )k^2\eta_2^2 <2 \alpha, \\[0.2cm]
& \alpha ^2\gamma_1^2-2\left(\chi+\frac{1}{k\eta_1}\right) \alpha \gamma_1+(1-2 \beta ) \alpha
+\chi^2<0,\\[0.2cm]
&2(1+ \beta )k^2\eta_1^2 \alpha ^2-4 \alpha  +\chi^2<0,\\[0.2cm]
&k\eta_2\gamma_2^2-2 \alpha \gamma_2+(1-2 \beta )k \alpha \eta_2<0.
\end{cases}
$$
In light of this,  \eqref{para-choice}  and \eqref{odi-requirement3}, we further compute that
\be\label{odi-requirement2}
\begin{cases}
&A_1:=-\left[(1+ \beta )\eta_2^2-\frac{2\alpha}{k^2}\right]=\frac{2}{k^2\alpha}(\alpha-\frac{\eta_2}{\sqrt{2}\eta_1})(\alpha+\frac{\eta_2}{\sqrt{2}\eta_1})>0,\\[0.2cm]
&A_2:=-\frac{1}{2}\left[(1-2 \beta )\alpha+(\gamma_1\alpha-\chi)^2-\frac{2\gamma_1\alpha}{k\eta_1} \right] =\frac{3}{2}\left(\frac{3+2k\eta_1\chi}{3k^2\eta_1^2}-\alpha\right)>0, \\[0.2cm]
&A_3:=-\frac{1}{2}\left[2(1+ \beta )\alpha\eta_1^2-\frac{4}{k^2}+\frac{\chi^2}{k^2\alpha}\right]=\frac{1}{k^2\alpha}\left(\alpha-
\frac{\chi^2}{2}\right)>0,\\[0.2cm]
&A_4:=-\frac{1}{2}\left(1-2 \beta  +\frac{\gamma_2^2}{\alpha}-\frac{2\gamma_2}{k\eta_2}\right)=\frac{\left(\eta_1^2\alpha^2-3k^2 \eta_1^2 \eta_2^2\alpha+2\eta_2^2\right)}{2k^2\eta_1^2\eta_2^2\alpha}>0.
\end{cases}
\ee
 With these preparations, we next apply Young's inequality multiple times to bound  the terms on the right-hand side of \eqref{diff-G} as follows:
\be\label{esti-1-pp0}
\begin{cases}
\frac{\eta_1\alpha}{k}\int_{\Omega}(U-1)(W-1)\leq \frac{\alpha}{2k^2}\int_{\Omega}(U-1)^2+ \frac{\eta_1^2\alpha}{2} \int_{\Omega}(W-1)^2,\\[0.2cm]
\frac{\eta_2}{k}\int_{\Omega}(U-1)(W-1)\leq  \frac{\eta_2^2}{2} \int_{\Omega}(U-1)^2+\frac{1}{2k^2}\int_{\Omega}(W-1)^2,\\[0.2cm]
-\frac{\tau_1\alpha}{k}\int_{\Omega}(U-1)V_t\leq\frac{\tau^2_1\alpha}{2}\int_{\Omega}V_t^2
+\frac{\alpha}{2k^2}\int_{\Omega}(U-1)^2
\end{cases}
\ee
and further
\be\label{esti-1-pp}
\begin{cases}
-\frac{\alpha}{k}\int_{\Omega}(U-1)V\leq\frac{\alpha}{2k^2}\int_{\Omega}(U-1)^2
+\frac{\alpha}{2}\int_{\Omega}V^2,\\[0.2cm]
-\frac{\tau_2}{k}\int_{\Omega}(W-1)Z_t\leq \frac{\tau^2_2}{2}\int_{\Omega}Z_t^2+
\frac{1}{2k^2}\int_{\Omega}(W-1)^2,\\[0.2cm]
-\frac{1}{k}\int_{\Omega}(W-1)Z\leq\frac{1}{2k^2}\int_{\Omega}(W-1)^2
+\frac{1}{2} \int_{\Omega}Z^2,\\[0.2cm]
\tau_2\eta_2\int_{\Omega}(U-1)Z_t\leq \frac{\tau_2^2}{2}\int_{\Omega}Z_t^2+\frac{\eta_2^2}{2}\int_{\Omega}(U-1)^2
\end{cases}
\ee
as well as
\be\label{esti-2-pp}
\begin{cases}
\frac{\gamma_2}{k}\int_{\Omega}(U-1)Z\leq\frac{\alpha}{2k^2}\int_{\Omega}(U-1)^2
+ \frac{\gamma_2^2}{2\alpha} \int_{\Omega}Z^2, \\[0.2cm]
\left(\eta_1\alpha-\frac{\chi}{k}\right)\tau_1\int_{\Omega}(W-1)V_t\leq \frac{\tau_1^2\alpha}{2}\int_\Omega V_t^2+\frac{1}{2\alpha}\left(\eta_1\alpha-\frac{\chi}{k}\right)^2\int_{\Omega}(W-1)^2,\\[0.2cm]
\frac{(\gamma_1\alpha-\chi)}{k}\int_{\Omega}(W-1)V\leq \frac{1}{2k^2}\int_\Omega (W-1)^2+\frac{(\gamma_1\alpha-\chi)^2}{2}\int_{\Omega}V^2.
\end{cases}
\ee
Moreover, from the second and fourth equation in \eqref{equal system}, we deduce that
\be\label{VZ-l2-by-pp}
\begin{cases}
2\tau_1\int_{\Omega}VV_t+2\int_{\Omega}|\nabla V|^2+\int_{\Omega}V^2\leq\eta^2_1\int_{\Omega}(W-1)^2,\\[0.25cm]
2\tau_2\int_{\Omega}ZZ_t+2\int_{\Omega}|\nabla Z|^2+\int_{\Omega}Z^2\leq\eta^2_2\int_{\Omega}(U-1)^2.
\end{cases}
\ee
Collecting \eqref{esti-1-pp0}, \eqref{esti-1-pp}, \eqref{esti-2-pp}, \eqref{VZ-l2-by-pp} and \eqref{diff-G}, we obtain that
\begin{equation*}
\begin{split}
&\mathcal{H}'(t)+\frac{4\alpha}{k}\int_\Omega|\nabla U^{\frac{1}{2}}|^2+\frac{4}{k}\int_\Omega|\nabla W^{\frac{1}{2}}|^2\\
&\ \ \ \ \ \  +\left(\frac{\gamma_1}{k\eta_1}+ 2\beta \right)\alpha\int_\Omega|\nabla V|^2+\left(\frac{\gamma_2}{k\eta_2}+ 2\beta \right)\int_\Omega|\nabla Z|^2\\
&\ \ \leq \left[\frac{2\alpha}{k^2}+(1+ \beta ) \eta_2^2\right]\int_\Omega (U-1)^2+\frac{1}{2}\left[(1-2 \beta )\alpha+(\gamma_1\alpha-\chi)^2-\frac{2\gamma_1\alpha}{k\eta_1}\right]\int_\Omega V^2\\
&\ \  +\frac{1}{2}\left[(1+2 \beta )\alpha\eta_1^2+\frac{4}{k^2}+\frac{1}{\alpha}\left(\eta_1\alpha
-\frac{\chi}{k}\right)^2+\frac{2\eta_1\chi}{k}\right]\int_\Omega(W-1)^2\\
&\ \ \ \   +\frac{1}{2}\left(1-2 \beta  +\frac{\gamma_2^2}{\alpha}-\frac{2\gamma_2}{k\eta_2}\right)\int_\Omega Z^2,
\end{split}
\end{equation*}
which together with \eqref{UW-l2} and \eqref{odi-requirement2} enables us to derive a key ODI for $\mathcal{H}$:
\begin{equation}\label{diff-G-lem1}
\begin{split}
\mathcal{H}'(t) &+\left(\frac{\gamma_1}{k\eta_1}+ 2\beta \right)\alpha\int_\Omega|\nabla V|^2+\left(\frac{\gamma_2}{k\eta_2}+ 2\beta \right)\int_\Omega|\nabla Z|^2\\
&\leq \left[(1+ \beta ) \eta_2^2-\frac{2\alpha}{k^2}\right]\int_\Omega (U-1)^2\\
& \ \ +\frac{1}{2}\left[(1-2 \beta )\alpha+(\gamma_1\alpha-\chi)^2-\frac{2\gamma_1\alpha}{k\eta_1}\right]
\int_\Omega V^2\\
& \ \  +\frac{1}{2}\left[2(1+ \beta )\alpha\eta_1^2-\frac{4}{k^2}+\frac{\chi^2}{k^2\alpha}\right]\int_\Omega(W-1)^2 \\
&\ \  +\frac{1}{2}\left(1-2 \beta  +\frac{\gamma_2^2}{\alpha}-\frac{2\gamma_2}{k\eta_2}\right)\int_\Omega Z^2\\
&=-A_1\int_\Omega(U-1)^2- A_2\int_\Omega V^2-A_3\int_\Omega(W-1)^2-A_4\int_\Omega Z^2.
\end{split}
\end{equation}
Now, writing
\be\label{decay-rate}
\begin{split}
\delta=\min\Bigr\{ & \frac{A_1k}{\alpha}, \  \frac{2A_2}{\alpha\tau_1\left
(1+2 \beta +\frac{\gamma_1}{k\eta_1}\right)},  \ A_3 k,  \\
&\  \frac{2A_4}{\tau_2\left
(1+2 \beta +\frac{\gamma_2}{k\eta_2}\right)}, \  2\left(\frac{\gamma_1}{k\eta_1\tau_1}+ \frac{2\beta}{\tau_1} \right), \   2\left(\frac{\gamma_2}{k\eta_2\tau_2}+ \frac{2\beta}{\tau_2} \right)\Bigr\},
\end{split}
\ee
and recalling  that  $\bar U=\bar W=1$ , from \eqref{UlnU+WlnW-by UWl2},  \eqref{diff-G-lem1} and \eqref{energy function}, we finally derive a simple ODI for $\mathcal{H}$ of the form:
$$
\mathcal{H}'(t)\leq -\delta \mathcal{H}(t), \quad t\in [ 0, T_m),
$$
which trivially gives rise to the desired exponential decay estimate \eqref{decay property}.
\end{proof}
\begin{remark}\label{4-2-lt} By setting $\chi=0, (\eta_1, \tau_1)=(\eta_2,  \tau_2), (W_0, \tau_2 Z_0)=(U_0, \tau_1 V_0)$, we see by uniqueness that \eqref{equal system} reduces to two copies of the minimal KS model:
\begin{eqnarray}\label{KS2-equal}
\begin{cases}
U_t =  \nabla\cdot\left(\nabla U-U\nabla V\right) &\text{in } \Omega\times(0,\infty), \\[0.2cm]
\tau_1 V_t = \Delta V-V+\eta_1(U-1)  &\text{in } \Omega\times(0,\infty), \\[0.2cm]
\frac{\partial U}{\partial \nu}=\frac{\partial V}{\partial \nu}=0&\text{on } \partial\Omega\times(0,\infty), \\[0.2cm]
\left(U, \   \tau_1V\right)=\left(U_0,\   \tau_1 V_0\right)&\text{in } \Omega\times\{0\}.
\end{cases}
\end{eqnarray}
These together with our subsequent $W^{j,\infty} (j\geq 1)$-convergence offer an exponential decay for  \eqref{KS2-equal}  with  convergence rates, extending and detailing  those of \cite{GZ98, Win10-JDE}.
\end{remark}
Given  the crucial starting $L^1$-convergence of $(U\ln U, W\ln W)$ provided in Lemmas \ref{con-lem-pe} and \ref{con-lem-pp}, which simply yields the uniform $L^1$-boundedness of $(U\ln U, W\ln W)$, repeating the arguments on boundedneness in our Section 3, especially Lemma \ref{gradUW-improve-lem}, we see that the solution to \eqref{equal system} or equivalently \eqref{Loop equations}  exists globally in time and is bounded in $L^\infty(\Omega\times (0, \infty))$; moreover, in the sprit of Lemma \ref{gradUW-improve-lem},  we have the following uniform higher order gradient estimate away from $t=0$, for some $C>0$,
\be\label{UVWZ-full bdd}
\begin{split}
&\| \left(U(t), \  W(t)\right)\|_{W^{2,4}}+\|\left(U(t), \ W(t)\right)\|_{W^{1,\infty}}\\[0.2cm]
 &+\| \left(I_{\{\tau_1=0\}}V(t), \ I_{\{\tau_2=0\}}Z(t)\right)\|_{W^{4,4}}+\|\left(V(t),
\ Z(t)\right)\|_{W^{3,\infty}}\leq C, \ \forall t\geq 1.
\end{split}
\ee

\begin{lemma}\label{con-improve-UWLp}  For any $p\geq 1$, there exists  $C_p=C(p, \eta_1,\eta_2, \tau_1,\tau_2,\chi, k)>0$ such  that
\begin{equation}\label{UW-LP-con}
\begin{split}
\|U(t)-1\|_{L^p}+\|W(t)-1\|_{L^p}\leq C_pe^{-\frac{\sigma}{2p} t}, \ \  t>0.
\end{split}
\end{equation}
 There exist constants $D_i>0$ depending on $(\eta_1,\eta_2,\chi, k)$ such that
\be\label{VZ-L2-con}
\begin{cases}
\|V(t)\|_{L^1} +\|Z(t)\|_{L^1}\leq   D_1 e^{- \frac{\sigma}{4} t},  \ t>0, &\text{ if } \tau_1=\tau_2=0, \\[0.2cm]
\left(\|V(t)\|_{L^1},\ \ \|Z(t)\|_{L^1}\right)\\[0.2cm]
\ \leq  D_2 \left(e^{-\frac{1}{3}\min\left\{\frac{1}{\tau_1}, \   \frac{\sigma}{2}\right\} t}, \ e^{-\frac{1}{3}\min\left\{ \frac{1}{\tau_2}, \  \frac{\sigma}{2}\right\} t}\right), \ t>0, &\text{ if } \tau_1,  \tau_2>0.
 \end{cases}
\ee
We remind here again $\sigma=\mu$ if $\tau_1=\tau_2=0$ and $\sigma=\delta$ if $\tau_1,  \tau_2>0$.
\end{lemma}
\begin{proof} In view of  the  Csisz\'{a}r-Kullback-Pinsker inequality (cf. \cite{CJMTU01}) and the facts that $\bar U=1=\bar W$ and \eqref{uw-ln-exp} or \eqref{ulnu-wlnw-gradv-z-con},  we infer, for some $C_1, C_2>0$, that
$$
\begin{cases}
&\|U-\overline{U}\|^2_{L^1}\leq 2\overline{U}\int_{\Omega}U\ln \frac{U}{\overline{U}}
=2\int_\Omega U\ln U\leq C_1e^{-\sigma  t}, \quad t>0, \\[0.25cm]
&\|W-\overline{W}\|^2_{L^1}\leq 2\overline{W}\int_{\Omega}W\ln \frac{W}{\overline{W}}
=2\int_\Omega W\ln W\leq C_2e^{-\sigma  t}, \quad t>0.
\end{cases}
$$
Hence, for any $p\geq 1$, the $L^\infty$-boundedness of $(U, W)$  provides some $C_3, C_4>0$ depending on $p$, $\eta_i, \tau_i, \chi$ and $\Omega$ such that
 \be\label{UW-Lp}
\begin{cases}
\|U-1\|_{L^p}\leq\|U-1\|^{\frac{p-1}{p}}_{L^{\infty}}\|U-1\|^{\frac{1}{p}}_{L^1}
\leq C_3e^{-\frac{\sigma}{2p} t},  \quad t>0, \\[0.25cm]
\|W-1\|_{L^p}\leq\|W-1\|^{\frac{p-1}{p}}_{L^{\infty}}\|W-1\|^{\frac{1}{p}}_{L^1}
\leq C_4e^{-\frac{\sigma}{2p} t},  \quad t>0.
 \end{cases}
\ee
Now, by the $V$-equation in \eqref{equal system}, we have
\be\label{V-Lq}
\begin{split}
\tau_1 \frac{d}{dt}\int_\Omega V^2+2\int_\Omega |\nabla V|^2+\int_\Omega V^2\leq\eta_1^2\int_\Omega (W-1)^2.
\end{split}
\ee
Thus, when $\tau_1=0$,  we  deduce from \eqref{V-Lq} and \eqref{UW-Lp} that
\be\label{V-L2-0}
\|V(t)\|_{L^1}\leq |\Omega|^\frac{1}{2} \|V(t)\|_{L^2}\leq \eta_1|\Omega|^\frac{1}{2} \|W-1\|_{L^2}\leq C_5e^{-\frac{\sigma}{4} t},  \ \ t>0;
\ee
and, when $\tau_1>0$, we get \eqref{V-Lq} and \eqref{UW-Lp} that
$$
\frac{d}{dt}\int_\Omega V^2+\frac{1}{\tau_1}\int_\Omega V^2\leq C_6 e^{-\frac{\sigma}{2}t},
$$
which enables us to derive that
\be\label{V-L2-1}
\|V(t)\|_{L^1}\leq |\Omega|^\frac{1}{2} \|V(t)\|_{L^2}\leq   C_7e^{-\frac{1}{3}\min\left\{\frac{1}{\tau_1}, \  \frac{\sigma}{2}\right\} t},  \ \ t>0.
\ee
One can easily use the same argument to the $W$-equation in \eqref{equal system} to infer that
\be\label{Z-L2}
\begin{cases}
 \|V(t)\|_{L^1}\leq |\Omega|^\frac{1}{2} \|V(t)\|_{L^2}\leq   C_8e^{- \frac{\sigma}{4} t},  \ t>0, &\text{ if } \tau_2=0, \\[0.2cm]
\|V(t)\|_{L^1}\leq |\Omega|^\frac{1}{2} \|V(t)\|_{L^2}\leq   C_9e^{-\frac{1}{3}\min\left\{\frac{1}{\tau_2}, \  \frac{\sigma}{2}\right\} t}, \ t>0, &\text{ if } \tau_2>0.
 \end{cases}
\ee
Now, our decay estimate \eqref{VZ-L2-con} follows trivially from \eqref{V-L2-0}, \eqref{V-L2-1} and \eqref{Z-L2}.
\end{proof}
At this position, based on the exponential decay estimate \eqref{UW-LP-con} and \eqref{VZ-L2-con}, one can use (commonly used,  cf. \cite{A19, A20, Win10-JDE}) the standard $W^{2,p}$-estimate in the case of $\tau_1=\tau_2=0$ or the $L^p$-$L^q$-smoothing estimate for the Neumann heat semigroup $e^{t\Delta}$ in the case of $\tau_1, \tau_2>0$ to derive the exponential decay of bounded solutions in up to $L^\infty$-norm. Here, thanks to our uniform higher order gradient estimates as in  \eqref{UVWZ-full bdd}, instead, we can easily apply  the G-N interpolation inequality to  improve the $L^p$-convergence to $W^{j,\infty} (j\geq 1)$-convergence of $(U, V, W, Z)$.

\begin{lemma}\label{con-fin-lem} Under Lemma \ref{con-lem-pe} or \ref{con-lem-pp},  there exist constants $L_i >0$ depending on $(\eta_1,\eta_2, \tau_1,\tau_2, \chi,  k)$ such that
\be\label{UW-com-con}
\begin{cases}
&\|\left(U(t)-1, \ \nabla U(t),  \  W(t)-1, \  \nabla W(t)\right)\|_{L^\infty}\\[0.2cm]
& \leq  L_1\left(e^{-\frac{\sigma}{6}t}, \ e^{-\frac{\sigma}{14}t}, \ e^{-\frac{\sigma}{6}t}, \ e^{-\frac{\sigma}{14}t}\right),\quad  \forall t\geq 1;
\end{cases}
\ee
and, in the case of $\tau_1=\tau_2=0$,
\be\label{VZ-com-con}
\begin{split}
&\|\left(V, \ \nabla V, \ D^2V, D^3 V\right)\|_{L^\infty}+\|\left(Z, \ \nabla Z, \ D^2Z, D^3 Z\right)\|_{L^\infty}\\[0.2cm]
&\leq L_2 \left(e^{-\frac{\mu}{12}t}, \ e^{-\frac{\mu}{16}t}, \ e^{-\frac{\mu}{20}t}, \ e^{-\frac{\mu}{44}t}\right), \quad    \forall t\geq 1;
\end{split}
\ee
in the case of $\tau_1, \tau_2>0$,
\be\label{VZ-com-con-pp}
\begin{cases}
\|\left(V, \ \nabla V, \ D^2V\right)\|_{L^\infty}\leq L_3 \left(e^{-\frac{\zeta_1}{9}t}, \ e^{-\frac{\zeta_1}{12}t}, \ e^{-\frac{\zeta_1}{15}t}\right), \quad \forall t\geq 1, \\[0.2cm]
\|\left(Z, \ \nabla Z, \ D^2Z\right)\|_{L^\infty}\leq L_4\left(e^{-\frac{\zeta_2}{9}t}, \ e^{-\frac{\zeta_2}{12}t}, \ e^{-\frac{\zeta_2}{15}t}\right), \quad  \forall t\geq 1.
\end{cases}
\ee
Here, $\zeta_i=\min\{\frac{1}{\tau_i}, \   \frac{\sigma}{2}\},  \ i=1,2$ and $\sigma$ is defined in Lemma \ref{con-improve-UWLp}.
\end{lemma}

\begin{proof}
Based on \eqref{UVWZ-full bdd}, \eqref{UW-LP-con} and the 2D G-N inequality in \eqref{G-N-I}, we deduce that
\be\label{U-Linfty-con}
\begin{split}
\|U-1\|_{L^\infty}&\leq C_1\|\nabla U\|_{L^\infty}^\frac{2}{3}\|U-1\|_{L^1}^\frac{1}{3}+C_1\|U-1\|_{L^1}\\[0.2cm]
&\leq C_2\|U-1\|_{L^1}^\frac{1}{3}\leq C_3e^{-\frac{\sigma}{6}t},\ \ t\geq 1.
\end{split}
\ee
and
\be\label{gradU-L-infty-con}
\begin{split}
\|\nabla(U-1)\|_{L^\infty}&\leq C_4\|D^2U\|_{L^4}^{\frac{6}{7}}\|U-1\|_{L^1}^{\frac{1}{7}}+C_4\|U-1\|_{L^1}\\
&\leq C_5\|U-1\|_{L^1}^\frac{1}{7}\leq C_6e^{-\frac{\sigma}{14}t},  \ \  t\geq 1.
\end{split}
\ee
The same reasonings applied to the $W$-component give us that
\be\label{WgradW-con}
\begin{cases}
\|W-1\|_{L^\infty}\leq C_7e^{-\frac{\sigma}{6}t},\ \ t\geq 1, \\[0.2cm]
 \|\nabla(W-1)\|_{L^\infty}\leq C_8e^{-\frac{\sigma}{14}t}, \ \    t\geq 1.
\end{cases}
\ee
The $W^{1,\infty}$-decay of $(U, W)$  in \eqref{UW-com-con} follows from \eqref{U-Linfty-con}, \eqref{gradU-L-infty-con} and \eqref{WgradW-con}.

When $\tau_1=\tau_2=0$,  we  use \eqref{UVWZ-full bdd}, \eqref{VZ-L2-con} and the 2D G-N inequality to infer that
\be\label{V-com-con}
\begin{cases}
\|V\|_{L^\infty}\leq C_9\|\nabla V\|_{L^\infty}^\frac{2}{3}\|V\|_{L^1}^\frac{1}{3}+C_9\|V\|_{L^1}\leq C_{10}e^{-\frac{\sigma}{12}t},\ \ t\geq 1, \\[0.2cm]
\|\nabla V\|_{L^\infty}\leq C_{11}\|D^2 V\|_{L^\infty}^\frac{3}{4}\|V\|_{L^1}^\frac{1}{4}+C_{11}\|V\|_{L^1}\leq C_{12}e^{-\frac{\sigma}{16}t},\ \ t\geq 1, \\[0.2cm]
\|D^2 V\|_{L^\infty}\leq C_{13}\|D^3 V\|_{L^\infty}^\frac{4}{5}\|V\|_{L^1}^\frac{1}{5}+C_{13}\|V\|_{L^1}\leq C_{14}e^{-\frac{\sigma}{20}t},\ \ t\geq 1,  \\[0.2cm]
\|D^3 V\|_{L^\infty}\leq C_{15}\|D^4 V\|_{L^4}^\frac{10}{11}\|V\|_{L^1}^\frac{1}{11}+C_{15}\|V\|_{L^1}\leq C_{16}e^{-\frac{\sigma}{44}t},\ \ t\geq 1.
\end{cases}
\ee
In a similar way via replacing $V$ by $Z$ in \eqref{V-com-con}, we obtain that
\be\label{Z-com-con}
\|\left(Z, \ \nabla Z, \ D^2Z, D^3 Z\right)\|_{L^\infty}\leq C_{17} \left(e^{-\frac{\sigma}{12}t}, \ e^{-\frac{\sigma}{16}t}, \ e^{-\frac{\sigma}{20}t}, \ e^{-\frac{\sigma}{44}t}\right), \quad t\geq 1.
\ee
Then the $W^{3,\infty}$-decay of $(V, Z)$ follows from \eqref{V-com-con} and \eqref{Z-com-con} by recalling $\sigma=\mu$.

When  $\tau_1, \tau_2>0$,  in a similar way to \eqref{V-com-con},  we  readily apply  \eqref{UVWZ-full bdd}, \eqref{VZ-L2-con} and the 2D G-N inequality in \eqref{G-N-I} to derive the $W^{2,\infty}$-decay of $(V, Z)$ in  \eqref{VZ-com-con-pp}.
\end{proof}
\begin{proof}[Proof of the $W^{j,\infty}$-exponential convergence in (B3)]
 Lemma \ref{con-fin-lem} actually proves more detailed exponential convergence about each order derivative of solution components than what has been  stated in \eqref{lt-thm} of  (B3). Here, we present a short proof of (B3). Indeed, using the transformations in \eqref{trans-id} that links \eqref{equal system} with \eqref{Loop equations} and translating Lemmas \ref{con-lem-pe}, \ref{con-lem-pp} and \ref{con-fin-lem} back to our original model \eqref{Loop equations}, we obtain the $W^{1,\infty}$-exponential convergence for $(u, w)$ as in \eqref{lt-thm} of  (B3). As for $(v, z)$, noticing the facts from \eqref{uvwz-l1} that
\begin{equation*}
\begin{split}
&\left\|\left(\bar v(t)-\bar{w}_0, \ \bar z(t)-\bar{u}_0\right)\right \|_{W^{j,\infty}}\\[0.2cm]
&=\begin{cases} (0, 0), & \text{ if } \tau_1=\tau_2=0, \\[0.2cm]
\left(\left|\bar{v}_0-\bar{w}_0\right| e^{-\frac{t}{\tau_1}}, \ \left|\bar{z}_0-\bar{u}_0\right| e^{-\frac{t}{\tau_2}}\right), & \text{ if } \tau_1, \tau_2>0,
\end{cases}
\end{split}
\end{equation*}
and then, in the case of $\tau_1, \tau_2>0$,  using \eqref{trans-id} and \eqref{VZ-com-con-pp}, we estimate
\be\label{zv-con-fin}
\begin{split}
&\left\|\left(\chi_1(v(t)-\bar{w}_0), \ \chi_2(z(t)-\bar{u}_0)\right)\right\|_{W^{2,\infty}}\\
&\leq
\left\|\left(\chi_1(v(t)-\bar{v}), \ \chi_2(z(t)-\bar{z})\right)\right\|_{W^{2,\infty}}+
\left\|\left(\chi_1(\bar v-\bar{w}_0), \chi_2(\bar z-\bar{u}_0)\right)\right\|_{W^{2,\infty}}\\
&\leq \left\|(V(t), \  Z(t))\right\|_{W^{2,\infty}}+C_{18}\left(e^{-\frac{1}{\tau_1}t}, \ e^{-\frac{1}{\tau_2}t}\right)\\
&\leq C_{19} \left(e^{-\frac{\zeta_1}{15}t}, \ e^{-\frac{\zeta_2}{15}t}\right), \ \ \ \forall t\geq 1.
\end{split}
\ee
In  the simple  case of $\tau_1=\tau_2=0$, we have from \eqref{VZ-com-con} that, for $t\geq 1$,
\be\label{vz-con-ee}
\left\|\left(\chi_1(v(t)-\bar{w}_0), \ \chi_2(z(t)-\bar{u}_0)\right)\right\|_{W^{3,\infty}}=\left\|(V(t), \  Z(t))\right\|_{W^{3,\infty}}\leq C_{20}  e^{-\frac{\mu}{44}t}.
\ee
Then our claimed $W^{j,\infty}(j=2,3)$-exponential convergence for $(v, z)$ in \eqref{lt-thm} of (B3) in the Introduction follows directly from  \eqref{zv-con-fin} and  \eqref{vz-con-ee}. \end{proof}

\textbf{Acknowledgments}  The authors  thank the anonymous referee very much  for carefully reading  our manuscript and giving positive and valuable comments, which further helped them to improve the exposition  of this work.   K. Lin is supported by the NSF of China  (No. 11801461), and T. Xiang   is supported by  the NSF of China  (No. 11601516 and 11871226) and the  Research Funds of Renmin University of China (No. 2018030199).


\begin{thebibliography}{99}
 \footnotesize

\bibitem{ADN59} S. Agmon, A. Douglis and L. Nirenberg, Estimates near the boundary for solutions of elliptic partial differential equations satisfying general boundary conditions, I, Commun. Pure Appl. Math., 12 (1959), 623--727.
\bibitem{ADN64} S. Agmon, A. Douglis and L. Nirenberg, Estimates near the boundary for solutions of elliptic partial differential equations satisfying general boundary conditions, II, Commun. Pure Appl. Math., 17 (1964),  35--92.

\bibitem{A16} N. Bellomo, A. Bellouquid, Y. Tao and M. Winkler, Toward a mathematical theory of Keller-Segel models of pattern formation in biological tissues, Math. Models Methods Appl. Sci., 25 (2015),  1663--1763.

\bibitem{TB17} T. Black,
Global existence and asymptotic stability in a competitive two-species chemotaxis system with two signals, Discrete Contin. Dyn. Syst. Ser. B,  22 (2017),  1253--1272.



\bibitem{CJMTU01} J.  Carrillo, A. \"{J}ungle, P. Markowich, G. Toscani and A. Unterreiter, Entropy dissipation methods for degenerate parabolic problems and generalized Sobolev inequalities, Monatsh. Math., 133 (2001), 1--82.

\bibitem{Cao15} X. Cao, Global bounded solutions of the higher-dimensional Keller-Segel system under smallness conditions in optimal spaces, Discrete Contin. Dyn. Syst.,   35  (2015),   1891--1904.

\bibitem{A7} C. Conca, E. Espejo and  K. Vilches, Remarks on the blowup and global existence for a two species chemotactic Keller-Segel system in $\mathbb{R}^2$, European J. Appl. Math.,  22 (2011), 553--580.

\bibitem{ESV09} E. Espejo Arenas,  A. Stevens and J. Velzquez, Simultaneous finite time blow-up in a two-species model for chemotaxis, Analysis (Munich),  29 (2009),  317--338.



\bibitem{Fri-bk}  A. Friedman, Partial Differential Equations, Holt, Rinehart Winston, New York, 1969.

\bibitem{GZ98} H. Gajewski and K. Zacharias, Global behaviour of a reaction-diffusion system modelling chemotaxis, Math. Nachr., 195 (1998), 77--114.

\bibitem{HV97} M. Herrero and  J. Vel\'{a}zquez,   A blow-up mechanism for a chemotaxis model,
Ann. Scuola Norm. Sup. Pisa Cl. Sci.,   24 (1997), 633--683.


\bibitem{HW01} D. Horstmann and G. Wang, Blow-up in a chemotaxis model without symmetry assumptions, European J. Appl. Math.,  12 (2001),  159--177.

\bibitem{Horst-03}  D. Horstmann,  From 1970 until present: the Keller-Segel model in chemotaxis and its consequences, I, Jahresber. Deutsch. Math. Verien.,    105 (2003), 103--165.


\bibitem{A21} D. Horstmann and M.  Winkler, Boundedness vs. blow-up in a chemotaxis system,  J. Differential Equations,  215 (2005),  52--107.

\bibitem{A6} D. Horstmann, Generalizing the Keller-Segel model: Lyapunov functionals, steady state analysis, and blow-up results for multi-species chemotaxis models in the presence of attraction and repulsion between competitive interacting species, J. Nonlinear Sci.,  21 (2011),  231--270.


\bibitem{JL92} W. J\"{a}ger, S. Luckhaus, On explosions of solutions to a system of partial differential equations modelling chemotaxis, Trans. Amer. Math. Soc.,  329 (1992),  819-824.




\bibitem{JW16-JDE} H. Jin,  and Z.  Wang,  Boundedness, blowup and critical mass phenomenon in competing chemotaxis,  J. Differential Equations,  260 (2016),  162--196.

\bibitem{JX18-DCDSB} H. Jin and T. Xiang, Repulsion reeefects on boundedness in a quasilinear attraction-repuslsion chemotaxis model in higher diemsnions, Discrete Contin. Dyn. Syst. Ser. B, 23 (2018),  3071--3085.



\bibitem{A1} H. Kn\'{u}tsd\'{o}ttir, E. P\'{a}lsson and L. Edelstein-Keshet, Mathematical model of macrophage-facilitated breast cancer cells invasion, J. Theor. Biol.,  357 (2014),  184--199.

\bibitem{KS08} R. Kowalczyk and Z.  Szyma\'{n}ska,
On the global existence of solutions to an aggregation model,  J. Math. Anal. Appl.,  343 (2008),  379--398.

\bibitem{Ladyzenskaja710} O. A. Ladyzenskaja, V. A. Solonnikov and N. N. Ural'eva,  Linear and Quasi-linear Equations of Parabolic Type, Amer. Math. Soc. Transl. 23, AMS, Providence, RI, 1968.



 \bibitem{A5} X. Li and  Y. Wang, Boundedness in a two-species chemotaxis parabolic system with two chemicals, Discrete Contin. Dyn. Syst. Ser. B,  22 (2017),  2717--2729.

 \bibitem{LL16} Y.  Li and J. Lankeit,
Boundedness in a chemotaxis-haptotaxis model with nonlinear diffusion,
Nonlinearity,  29 (2016), 1564--1595.


\bibitem{A19} K. Lin and C.  Mu, Global existence and convergence to steady states for an attraction-repulsion chemotaxis system, Nonlinear Anal. Real World Appl.,  31 (2016),  630--642.



\bibitem{A20} K. Lin, C.  Mu and D.  Zhou, Stabilization in a higher-dimensional attraction-repulsion chemotaxis system if repulsion dominates over attraction, Math. Models Methods Appl. Sci.,  28 (2018), 1105--1134.


\bibitem{A2} K. Lin and  T. Xiang, On global solutions and blow-up for a short-ranged chemical signaling loop, J. Nonlinear Sci.,  29 (2019), 551--591.


\bibitem{A11} T. Nagai, Blow-up of radially symmetric solutions to a chemotaxis system, Adv. Math. Sci. Appl.,  5 (1995),  581--601.

\bibitem{A12} T. Nagai, T. Senba and  K. Yoshida, Application of the Trudinger-Moser inequality to a parabolic system of chemotaxis,  Funkcial. Ekvac.,  40 (1997),  411--433.

\bibitem{A13} T. Nagai, Blow-up of nonradial solutions to parabolic-elliptic systems modelling chemotaxis in twodimensional domains, J. Inequal. Appl., 6 (2001),  37--55.



\bibitem{PW60} L.  Payne and H.  Weinberger, An optimal Poincar\'{e} inequality for convex domains, Arch. Rational Mech. Anal.,  5 (1960),  286--292.

\bibitem{QG19} H. Qiu and S.  Guo,
Global existence and stability in a two-species chemotaxis system,  Discrete Contin. Dyn. Syst. Ser. B,  24 (2019),   1569--1587.

\bibitem{STW14} C. Stinner, J. Tello and M.  Winkler, Competitive exclusion in a two-species chemotaxis model,  J. Math. Biol.,  68 (2014),  1607--1626.



\bibitem{A14} Y. Tao and Z. Wang, Competing effects of attraction vs. repulsion in chemotaxis, Math. Models Methods Appl. Sci., 23 (2013),  1--36.


 \bibitem{TW11} Y. Tao and  M. Winkler,  A chemotaxis--haptotaxis model: the roles of nonlinear diffusion and logistic source, SIAM J. Math. Anal., 43 (2011), 685--704.

 \bibitem{TW14-JDE} Y.  Tao and M. Winkler, Energy-type estimates and global solvability in a two-dimensional chemotaxis-haptotaxis model with remodeling of non-diffusible attractant, J. Differential Equations,  257 (2014), 784--815.



\bibitem{A4} Y. Tao and M. Winkler, Boundedness vs. blow-up in a two-species chemotaxis system with two chemicals, Discrete Contin. Dyn. Syst. Ser. B,  20 (2015), 3165--3183.

\bibitem{TW12-non} J. Tello and M.  Winkler, Stabilization in a two-species chemotaxis system with a logistic source, Nonlinearity,  25 (2012), 1413--1425.


\bibitem{TMZK18} X. Tu,  C.  Mu,  P.  Zheng and K.  Lin,
Global dynamics in a two-species chemotaxis-competition system with two signals,  Discrete Contin. Dyn. Syst., 38 (2018),   3617--3636.



\bibitem{Win10-JDE} M. Winkler,  Aggregation vs. global diffusive behavior in the higher-dimensional Keller-Segel model, J. Differential Equations, 248 (2010), 2889--2905.

\bibitem{Win2013} M. Winkler, Finite-time blow-up in the higher-dimensional parabolic-parabolic Keller-Segel system, J. Math. Pures Appl.,  100  (2013), 748--767.

\bibitem{Xiang15-JDE}T. Xiang, Boundedness and global existence in the higher-dimensional parabolic-parabolic chemotaxis system with/without growth source,
J. Differential Equations, 258 (2015),  4275--4323.

\bibitem{Xiang18-NA} T. Xiang,  Global dynamics for a diffusive predator-prey model with prey-taxis and classical Lotka-Volterra kinetics,  Nonlinear Anal. Real World Appl.,  39 (2018), 278--299.

\bibitem{Xiang18-JMP} T. Xiang,  Sub-logistic source can prevent blow-up in the 2D minimal Keller-Segel chemotaxis system, J. Math. Phys., 59 (2018),  081502, 11 pp.


\bibitem{A3} H. Yu, W. Wang and  S. Zheng, Criteria on global boundedness versus finite time blow-up to a two-species chemotaxis system with two chemicals, Nonlinearity, 31 (2018),  502--514.

\bibitem{ZLY17} Q. Zhang,   X. Liu and X.  Yang,
Global existence and asymptotic behavior of solutions to a two-species chemotaxis system with two chemicals,  J. Math. Phys. 58 (2017),   111504, 9 pp.

\bibitem{Zhang19} Q. Zhang, Competitive exclusion for a two-species chemotaxis system with two chemicals, Appl. Math. Lett.,  83 (2018), 27--32.

\bibitem{ZC17} P. Zheng and C.  Mu, Global boundedness in a two-competing-species chemotaxis system with two chemicals, Acta Appl. Math.,  148 (2017), 157--177.

\end{thebibliography}
\end{document}